\documentclass{amsart}
\usepackage{a4,amsmath,bbm,amsthm,graphicx,hyperref}
\usepackage{txfonts}
\usepackage{enumitem}
\newcommand{\N}{\mathbbm{N}}
\newcommand{\R}{\mathbbm{R}}
\newcommand{\e}{\mathrm{e}}
\newcommand{\E}{\mathrm{E}}
\def\bu{{\bf u}}

\def\by{{\bf y}}
\def\bz{{\bf z}}
\def\bw{{\bf w}}
\def\bff{{\bf f}}
\def\bB{{\bf B}}
\def\bF{{\bf F}}

\def \reyn{ \boldsymbol {\sigma}^{(\textsc{r})}}
\newcommand{\mean}[1]{\overline{#1}}
\newtheorem{theorem}{Theorem}[section]
\newtheorem{corollary}{Corollary}[section]

\newtheorem{remark}{Remark}[section]
\def\Om{\Omega}
\newcommand{\g} {\nabla} 
\def \reyny{ \boldsymbol {\sigma}^{(\textsc{r})}_\by}
\def \reynz{ \boldsymbol {\sigma}^{(\textsc{r})}_\bz }

\newcommand{\ep}{\varepsilon}

\numberwithin{equation}{section}

\makeatletter
\renewcommand{\@biblabel}[1]{#1\hfill \hspace{-0.2cm}}
\makeatother

\begin{document}

\title{Long-Time Reynolds Averaging of Reduced Order Models for Fluid Flows: Preliminary Results}

\author{Luigi C. Berselli}   
\address{(Luigi C. Berselli) Department of Mathematics, Universit\`a di Pisa, Pisa,
  I56127, ITALY }
\email{luigi.carlo.berselli@unipi.it}
\author{Traian Iliescu}
\address{(Traian Iliescu) Department of Mathematics, Virginia Tech, Blacksburg, VA 24061, USA.}
\email{iliescu@vt.edu}
\author{Birgul Koc}
\address{(Birgul Koc) Department of Mathematics, Virginia Tech, Blacksburg, VA 24061, USA.}
\email{birgul@vt.edu}
\author{Roger Lewandowski}
\address{(Roger Lewandowski) Univ Rennes \& INRIA, CNRS - IRMAR UMR 6625 \& Fluminance team,
   Rennes, F-35042, France.}
\email{Roger.Lewandowski@univ-rennes1.fr}
%

\address{%
}


\begin{abstract}
  We perform a theoretical and numerical investigation of the time-average of energy
  exchange among modes of Reduced Order Models (ROMs) of fluid flows.  We are interested
  in the statistical equilibrium problem, and especially in the possible forward and
  backward average transfer of energy among ROM basis functions (modes).  We consider two
  types of ROM modes: eigenfunctions of the Stokes operator and Proper Orthogonal
  Decomposition (POD) modes.  We prove analytical results for both types of ROM modes and
  we highlight the differences between them.  We also investigate numerically whether the
  time-average energy exchange between POD modes is positive.  To this end, we utilize the
  one-dimensional Burgers equation as a simplified mathematical model, which is commonly
  used in ROM tests.  The main conclusion of our numerical study is that, for long enough
  time intervals, the time-average energy exchange from low index POD modes to high index
  POD modes is positive, as predicted by our theoretical results.
\end{abstract}

\maketitle
\section{Introduction}
In this note we combine some results on the long-time averaging of fluid equations with
the recently developed techniques for reduced order model (ROM) development. In this
preliminary work we start proving some analytical results that characterize the
time-averaged effect of the exchange of energy between various modes, both in the case of
the computable decomposition made with proper orthogonal decomposition (POD) type basis
functions and with the abstract basis made with eigenfunctions. We will show that the
results obtainable with a generic (but computable) basis are less precise than those
obtainable with the abstract spectral basis, the difference coming from the lack of
orthogonality of the gradients of the POD basis functions.

We then provide a few numerical examples. Concerning the analytical results we will prove
partial results for the energy exchange between large and small scales, showing the
difference between the use of a spectral type basis, versus a POD one.  In particular, we
are interested in results connected to the statistical equilibrium problem, which can be
deduced in a computable way by a long-time averaging of the solutions. We want to
investigate the possible forward and backward average transfer of energy. The properties
of a turbulent flow are computable (and relevant) only in an average sense. In this
respect, we want to follow the most classical approach dating back to Stokes, Reynolds,
and Prandtl of considering long-time averages of the solution as the key quantity to be
computed or observed. Therefore, we do not need to consider statistical averages and link
them with time averaging by means of (unproved) ergodic hypotheses.

To introduce the problem that we will consider, we recall that a Newtonian incompressible flow
(with constant density) can be described by the Navier-Stokes equations (NSE in the sequel)
\begin{equation} 
  \label{eq:NSE}
  \left\{\begin{aligned}
      \partial_t \bu - \nu\Delta \bu +(\bu \cdot \nabla)\, \bu + \nabla p &=\bff \qquad
      \mbox{in }\Omega\times (0,T),
      \\
      \nabla \cdot \bu&=0 \qquad \mbox{in }\Omega\times (0,T),
      \\
      \bu&=\mathbf{0} \qquad \mbox{on }\Gamma\times (0,T) ,
      \\
      \bu(\cdot, 0)&=\bu(0) \qquad \mbox{in }\Omega,
    \end{aligned}
  \right.
\end{equation}
with Dirichlet conditions when the motion takes place in a smooth and bounded domain
$\Omega\subset\R^3$ with solid walls $\Gamma:=\partial\Omega$. The unknowns are the
velocity field $\bu$ and the scalar pressure $p$, while the positive constant $\nu>0$ is
the kinematic viscosity. The key parameter to detect the nature of the problem is the
non-dimensional Reynolds number, which is defined as
\begin{equation*}
  Re=\frac{U L}{\nu} ,
\end{equation*}
where $U$ and $L$ are a characteristic velocity and length of the problem. In realistic
problems, the Reynolds number can be extremely large (in many cases of the order of
$10^{8}$, but up to the order of $10^{20}$ in certain geophysical problems).  For
simplicity in the notation, we use as a control parameter the viscosity and we assume that
it is very small. Hence, the effect of the regularization (similar to the diffusion in
heat transfer) due to the Laplacian is negligible and the behavior of solutions is really
turbulent and rather close to the motion of ideal fluids.  Due to the well-known
difficulties in performing direct numerical simulations (DNS), it is nowadays a
well-established technique that of trying to reduce the computational efforts by
simulating only the largest scales, which are nevertheless the only ones really observable
and the only ones which are needed in order to optimize macroscopic properties as, e.g.,
drag or lift in the craft design. In this framework, the large eddy simulation (LES)
methods, which emerged in the last 30 years, are among the most popular, and they found a
very relevant role within both theorists' and practitioners' communities. For recent LES
reviews see for instance the monographs~\cite{Sag2001,BIL2006, LR2012,CL2014}.

The LES methods are in many cases very well developed and both theoretically and
computationally appealing, especially for problems without boundaries, but many
difficulties and open problems arise when facing solid boundaries. In most cases the
design of efficient LES methods is guided by deep properties of the solutions, as emerging
from fine theorems of mathematical analysis. Furthermore, the ultimate goal of having a
golden standard is far from being obtained, and large families of methods (wave-number
asymptotics, differential filters, $\alpha$-models) attracted the interest of different
communities, spanning from the pure mathematicians, to the applied geophysicist and
mechanical engineers, as well as computational practitioners. For this reasons we believe
that it is important to have some well-defined and clearly stated guidelines, that can be
adapted to different problems. In this way the methods can be improved with insight not
only from experts in modeling, but also from mathematicians, physicists, and computational
scientists.

In this respect, we point out that very recently the use of other (more flexible and
computationally simpler) ways of finding approximate systems has become very popular. The
LES methods itself can be specialized or even glued with other ways of determining
approximate and much smaller systems, which are computable in a very efficient way.  For
instance, reduced order modeling is increasingly becoming an accepted paradigm, in which
applications to fluids are still being
developed~\cite{HRS2016,LMQR2014,QMN2016,QRM2011,Roz2014}.

The basic ansatz at the basis of the use of these models is the approximation of the
velocity by a truncation of the series
\begin{equation*}
  \bu=\sum_{k=1}^\infty \bu_k \bw_k,
\end{equation*}
where $\{\bw_k\}_{k\in\N}$ is a basis constructed by using the POD, not necessarily made
with eigenfunctions of the Stokes operator, and the coefficients of the $L^2$-projections
are evaluated as follows
\begin{equation*}
  \bu_k=\frac{\int_{\Omega}\bu\cdot\bw_k\,dx}{\int_\Omega|\bw_k|^2\,dx}.
\end{equation*}
The appealing property of this approach is that the choice of the basis is adapted-to and
determined-by the problem itself. Generally the basis is chosen after a preliminary
numerical computation, hence it contains at least the basic features of the solution and
geometry of the problem to be studied.  The other basic fact is that the kinetic energy of
the problem is the key quantity under consideration; in fact the $L^2$-projection is used
to determine the approximate velocity and also the energy content in the basis
construction. To determine the number $r\in \N$ such that the solution is projected over
the span generated by the (orthogonal) functions $V_r:=\text{span}\{\bw_1,\dots,\bw_r\}$,
generally it is assumed that the projection of the solution over $V_r$ contains a large
percentage (say 80$\%$) of the total kinetic energy of the underlying system.

It turns out that a basis associated with the problem at hand can greatly improve the
effectiveness of the ROM. Its proper choice can be of great interest in applications to
fluid flows~\cite{WWXI2017a,XWWI2017b,xie2018lagrangian}. The main goal and novelty of the
present paper is to try to combine results on the long-time behavior of fluid flows
(especially in the case of statistical equilibrium) with reduced order modeling. We use
this approach in order to capture both the long-time averages and the energy exchange of
the ROM solutions.  In this respect, we are extending to the POD setting the results based
on statistical solutions by Foias~\textit{et al.}~\cite{Foi197273,FMRT2001a} and those
more recently obtained for time-averages in~\cite{BL2018,Lew2015}. In this respect, the
main theoretical results of this paper, stated in Theorems~\ref{thm:newtheorem}
and~\ref{thm:newtheorem2} below, can be viewed as a spectral version of the results
of~\cite{BL2018,Lew2015}. It tells that low frequency modes yield a Reynolds stress which
is dissipative in the mean, the total spatial mean work of it being larger than the long
time average of the dissipation of the fluctuations, which is consistent with observations
and results in~\cite{BL2018, Foi197273}. However, the analysis shows that the triad
interaction between high and low frequency modes yields an additional non positive term in
the budget between the Reynolds stress of high modes and the corresponding mean
dissipation. This term may be non dissipative and may permit an inverse energy cascade,
which is not in contradiction with the fact that the total Reynolds stress is dissipative
in mean.

\medskip

\textbf{Plan of the paper.} The paper is organized as follows: In
Section~\ref{sec:ROM}, we outline the general framework for ROMs of fluid flows, and we
display the exchange of energy between large scales and small scales for two ROM bases:
the POD and the Stokes eigenfunctions.  In Section~\ref{sec:time-averages}, we present
some preliminaries on long-time averages.  In Section~\ref{sec:average-energy}, we prove
the main theoretical results for the average transfer of energy for ROMs constructed with
the POD and the Stokes eigenfunctions.  In Section~\ref{sec:numerical-results}, we
investigate the theoretical results in the numerical simulation of the one-dimensional
Burgers equation.  Finally, in Section~\ref{sec:conclusions}, we draw conclusions and
outline future research directions.
\section{Reduced Order Modeling}
\label{sec:ROM}
As outlined in the introduction, one key quantity in the pure and applied analysis of the
Navier-Stokes equations is the kinetic energy, since it is both a meaningful physical
quantity, but also the analysis of its budget is at the basis of the abstract existence
results (cf. Constantin and Foias~\cite{CF1988}) and also of the conventional turbulence
theories of Kolmogorov~\cite{Kol1941}. It is well-known that after testing the
NSE~\eqref{eq:NSE} by $\bu$ and integrating over the space-time, one (formally) obtains
the global energy balance
\begin{equation*}
  \frac{1}{2}\|\bu(t)\|^2+\nu\int_0^t\|\nabla
  \bu(s)\|^2\,ds=\frac{1}{2}\|\bu(0)\|^2+\int_0^t\int_\Omega \bff\cdot\bu\,dx ds.
\end{equation*}
At present, we are only able to prove that the above balance holds true with the sign of
``less or equal'' for the class of weak (or turbulent) solutions that we are able to
construct globally in time, without restrictions on the viscosity and on the size of
square summable initial velocity $\bu(0)$ and external force $\bff$.  It is of fundamental
importance in many problems in pure mathematics to understand under which hypotheses the
equality holds true.  We are now focusing on the ``global energy'' which is an averaged
quantity, since it is the integral of the square modulus of the velocity over the entire
domain. We also point out that at the other extreme one can deduce, without the
integration over the domain, the point-wise relation
\begin{equation*}
  \partial_t\frac{|\bu|^2}{2}+|\nabla \bu|^2-\Delta\frac{|\bu|^2}{2}+\text{div
  }\left(\frac{\bu|\bu|^2}{2}+p\,\bu\right)=\bff\cdot\bu.
\end{equation*}
In between there is the so called ``local energy'' which can be obtained by multiplying
the NSE by $\bu\,\phi$, where $\phi$ is a bump function, before integrating in the
space-time variables. The goal is to show that
\begin{equation*}
  \int_0^T\int_{\Omega}|\nabla \bu|^2\phi\,dx dt=
  \int_{0}^{T}\int_{\Omega}\left[\frac{|\bu|^2}{2}\left(\partial_t\phi+\Delta\phi\right)
    +\left(\frac{|\bu|^2}{2}+p\right)\bu\cdot\nabla \phi+\bff\cdot \bu\,\phi\right]\,dx dt
\end{equation*}
holds (at least with the inequality sign) for all smooth scalar functions $\phi\in
C^\infty_0((0,T)\times\Omega)$ such that $\phi\geq0$. The validity of such an inequality
is one of the requests to use the partial regularity results, but it is also one of the
requests to be satisfied by weak solutions constructed by numerical or LES methods. In
this respect, see Guermond, Oden and Prudhomme~\cite{GOP2004} and also~\cite{BFS2018}.

In this paper we study the global energy in the perspective that it can be reconstructed
in a computable way or it can be well approximated by the POD basis functions $\{\bw_k\}$.

The fact that the functions $\{\bw_k\}$ can be constructed to be orthonormal with respect to
the scalar product in $(L^2(\Omega)^{3},\,\|\cdot\|)$ allows us to evaluate the kinetic energy
easily by the following numerical series
\begin{equation*}
  \E(\bu)=\frac{1}{2}\sum_{k=1}^\infty\|\bu_k\|^2.
\end{equation*}
Since we are going to use only a reduced number of ROM modes, it is relevant to consider
the energy contained in functions described by a restricted set of indices. Hence,
following the notation in~\cite{FMRT2001a}, if we define
\begin{equation*}
  \bu_{m',m''}:=\sum_{k=m'}^{m''}\bu_k\bw_k,
\end{equation*}
then the kinetic energy content of $\bu_{m',m''}$ is simply evaluated as follows 
\begin{equation*}
  \E(\bu_{m',m''})=\frac{1}{2}\sum_{k=m'}^{m''}\|\bu_k\|^2.
\end{equation*}
We denote by $\mathbf{P}_m$ the projection operator over the subspace $V_m$ spanned by the
first $m$-functions
$\{\bw_k\}_{k=1,\dots,m}$ and we want to investigate the energy transfer between the various
modes, together with averaged long-time behavior associated with this splitting.

We are going to adapt well-known studies on the decomposition in small and large
eddies. This would be the case if the functions $\bw_k$ are chosen to be the
eigenfunctions of the Stokes operator, hence associated with large and small
frequencies. In our setting the basis is determined by the solution of a simplified
problem, which can be treated computationally, and the basis functions are orthonormal in
$L^2(\Omega)$, but we cannot expect that their gradients are also orthogonal.

\medskip

For the NSE, the standard ROM is constructed as follows:
\begin{enumerate}
\item[(i)] choose modes $\{ \bw_1, \ldots, \bw_d\}$, which represent the recurrent spatial
  structures of the given flow;
\item[(ii)] choose the dominant modes $\{\bw_1,\ldots,\bw_m\}$, with $m \leq d$, as basis
  functions for the ROM;
\item[(iii)] use a Galerkin truncation $\bu_m = \sum_{j=1}^{m} a_j \, \bw_j$;
\item[(iv)] replace $\bu$ with $\bu_m$ in the NSE; 
\item [(v)] use a Galerkin projection of NSE~($\bu_m$) onto the ROM space 
$V_m:=\mbox{span} \{\bw_1,\ldots,\bw_m\}$ to obtain
  a low-dimensional dynamical system, which represents the ROM:
  \begin{equation*}
    \dot{a}= A\,a+a^{\top}\,B\,a\, ,
  \end{equation*}
  where $a$ is the vector of unknown ROM coefficients and $A, B$ are ROM operators; 
\item[(vi)] in an offline stage, compute the ROM operators;
\item[(vii)] in an on-line stage, repeatedly use the ROM (for various parameter settings
  and/or longer time intervals).
\end{enumerate}
Hence, there is a very natural splitting of the velocity field $\bu$ into two components,
the part coherent with the basis expansion associated with the less energetic modes and
the remainder. This can be formalized as follows:
\begin{equation*}
  \bu=\by+\bz,
\end{equation*}
where, for $m\in \N$ which can be selected computationally (based, e.g., on the relative
kinetic energy content in the first $m$ POD-modes, but other choices relative to the
enstrophy are possible) in order to have a significant amount of the energy content of the
flow
\begin{equation*}
  \by=\sum_{k=1}^{m}\bu_k\bw_k=\mathbf{P}_m\bu\qquad\text{and}\qquad
  \bz=\sum_{k=m+1}^{+\infty}\bu_k\bw_k=(\mathbf{I}-\mathbf{P}_m)\,\bu=:\mathbf{Q}_m\,\bu.  
\end{equation*}
We observe that we are considering the functions $\{\bw_k\}$ as divergence-free. Generally
they are not ``exactly divergence-free,'' but numerically we can consider that they have
vanishing divergence, hence in the computations which will follow the pressure terms can
be dropped by a standard Leray projection. It will be nevertheless interesting to consider
also bases which are not divergence-free. Relevant results, due to the computational
simplifications, could be also derived by the combination of ROM with artificial
compressibility methods, as those introduced in~\cite{DLM2017,GMS2006}.

In addition, we consider the external force as stationary, that is $\bff=\bff(x)\in
L^2(\Omega)^{3}$ and we look for conditions holding at statistical equilibrium. Our purpose is
to determine --if possible-- the long-time behavior of $\by$ and to analyze the energy
budget between low and high modes in the orthogonal decomposition determined by the
functions $\bw_k$.

As usual in many problems fluid mechanics, we use the Hilbert space functional setting with
\begin{equation*}
  \begin{aligned}
    & \mathcal{V} = \{ \boldsymbol{\varphi} \in \mathcal{D}(\Om)^3, \, \, \g \cdot
  \boldsymbol{\varphi} = 0 \}, 
  \\
  &H =
  \left\{\bu\in L^2(\Omega)^3, \ \nabla\cdot \bu=0,\
    \bu\cdot\mathbf{n}=0 \text{ on }\Gamma\right\},
  \\
  &V = 
  \left\{\bu\in H^1_{0}(\Omega)^3, \,  \nabla\cdot \bu=0\right\},
\end{aligned}
\end{equation*}
where $\mathbf{n}$ denotes the outward normal unit vector. Moreover, $V'$ is the
topological dual space of $V$. We will also denote by $<\,.\, ,\,. \, >$  the duality
pairing between $V$ and $V'$. We recall that $\mathcal{V}$ is dense in $H$ and $V$ for their
respective topologies~\cite{GR86, JLL69}.

Once we project $L^2(\Omega)^3$ over the subspace $H$ of divergence-free and
tangential vector fields by means of the Leray projection operator $\mathrm{P}$, we have
the following abstract (functional) equation in $H$
\begin{equation*}
  \frac{d \bu}{dt}+\nu A\bu+B(\bu,\bu)=\mathrm{P}\,\bff,
\end{equation*}
where $A:=-\mathrm{P}\,\Delta$, while $B(\bu,\bu):=\mathrm{P} \,((\bu\cdot\nabla)\,\bu)$.
As usual in this analysis (see for instance~\cite{FMRT2001a}) we can start by assuming
that the input force can be decomposed within a finite sum of basis functions (or that it
belongs to $V_{m}$, which will be clarified in the next section section, in particular by
Theorem~\ref{thm:thm_proj}), hence
\begin{equation*}
  \mathbf{P}_m \,\bff=\bff.
\end{equation*}
We  split the Navier-Stokes equations into a coupled system for $\by\in
\mathbf{P}_m H$ and $\bz\in (\mathbf{P}_m H)^\perp=\mathbf{Q}_m H$ as follows 
\begin{equation}
\label{eq:system}
\begin{aligned}
    &\frac{d \by}{dt}-\nu\, \mathbf{P}_m (\Delta\bu)+\mathbf{P}_m
    B(\by+\bz,\by+\bz)=\mathbf{P}_m \,\bff,
    \\
    & \frac{d \bz}{dt}-\nu\,
    \mathbf{Q}_m(\Delta\bu)+\mathbf{Q}_m B(\by+\bz,\by+\bz)=\mathbf{0},  
  \end{aligned}
\end{equation}
where we used that both $\mathbf{P}_m$ and $\mathbf{Q}_m $ commute with the time
derivative.

Once we evaluate the kinetic energy, since $\mathbf{P}_m\by=\by$ and
$\mathbf{Q}_m\,\bz=\bz$ we get (by integrating by parts and by using the fact
that functions vanish at the boundary) that
\begin{equation*}
  \begin{aligned}
    & -\nu\int_\Omega \mathbf{P}_m (\Delta\bu)\cdot \by\,dx=-\nu\int_\Omega
    (\Delta\bu)\cdot \by\,dx=-\nu\int_\Omega (\Delta\by+\Delta\bz)\cdot \by\,dx=
    \nu\,\|\nabla\by\|^2+\nu\int_\Omega\nabla \by:\nabla\bz\,dx,
    \\
    & -\nu\int_\Omega \mathbf{Q}_m (\Delta\bu)\cdot \bz\,dx=-\nu\int_\Omega
    (\Delta\bu)\cdot \bz\,dx=-\nu\int_\Omega (\Delta\by+\Delta\bz)\cdot \bz\,dx=
    \nu\,\|\nabla\bz\|^2+\nu\int_\Omega\nabla \by:\nabla\bz\,dx.
  \end{aligned}
\end{equation*}
In this way we can obtain the following equality and inequality
\begin{equation}
  \label{eq:splitting_POD}
  \begin{aligned}
    &\frac{1}{2}\frac{d }{dt}\|\by\|^2+\nu \|\nabla
    \by\|^2+\nu(\nabla\by,\nabla\bz)+(B(\by+\bz,\by+\bz),\by)=(\bff,\by),
    \\
    & \frac{1}{2}\frac{d }{dt}\|\bz\|^2+\nu \|\nabla
    \bz\|^2+\nu(\nabla\by,\nabla\bz)+(B(\by+\bz,\by+\bz),\bz)\leq\mathbf{0},
  \end{aligned}
\end{equation}
These are the two basic balance equations that we will use to infer the behavior and
transfer of the kinetic energy between $\by$ and $\bz$. Notice that the balance relation
for $\by$, involving just a finite combination of rather smooth functions is an equality,
while the second one is an inequality. In fact, the second one can be derived by a
limiting argument and in the limit the lower semi-continuity of the norm will produce the
inequality.

Since the tri-linear term $(B(\bu,\bu),\bw)$ is skew-symmetric with respect to the last two
variables, we obtain from~\eqref{eq:splitting_POD}
\begin{equation} 
  \begin{aligned}
    &\frac{1}{2}\frac{d }{dt}\E(\by)+\nu \|\nabla
    \by\|^2+\nu(\nabla\by,\nabla\bz)=(B(\by,\by),\bz)-(B(\bz,\bz),\by)+(\bff,\by),
    \\
    & \frac{1}{2}\frac{d }{dt}\E(\bz)+\nu \|\nabla
    \bz\|^2+\nu(\nabla\by,\nabla\bz) \leq-(B(\by,\by),\bz)+(B(\bz,\bz),\by).
    \label{eqn:energy-equation-pod}
  \end{aligned}
\end{equation}
This is a formal setting, which is obviously true for strong solutions of the NSE, where
the inequality in~\eqref{eqn:energy-equation-pod} is an equality. When considering weak
solutions, the integral $(B(\bz,\bz),\bz)$ might be not defined in $L^1(0,T)$ for
regularity issues. However, one can still rigorously
derive~\eqref{eqn:energy-equation-pod} by a double frequency truncation or a
regularization of the operator $B$ by considering $(B(\bz\star\rho_\varepsilon, \bz),\bz)$
for a given standard mollifier $\rho_\ep$ and passing to the limit when $\ep \to
0$. Details are standard and out of the scope of the present paper.

We observe that $-(B(\by,\by),\bz)$ is the energy flux induced in the more energetic terms
by the inertial forces associated to less energetic modes, while $-(B(\bz,\bz),\by)$ is
the energy flux induced in the less energetic terms by the inertial forces associated to
more energetic modes.  In a schematic way we can decompose the rate of transfer of kinetic
energy $\mathrm{e}_m(\bu)$ into two terms as follows
\begin{equation} 
  \label{eq:em} 
  \mathrm{e}_m(\bu):= e^{\uparrow}(\bu)- e^{\downarrow}(\bu)\qquad \text{ with}\qquad
  e^{\uparrow}(\bu):=-(B(\by,\by),\bz),\qquad   e^{\downarrow}(\bu):=(B(\bz,\bz),\by). 
\end{equation}
We also use the following notation: 
\begin{equation} 
  \label{eq:Em} 
  \mathcal{E}_{m}(\bu) := - \nu(\nabla\by,\nabla\bz). 
\end{equation}

Hence, we can rewrite~\eqref{eqn:energy-equation-pod} as follows 
\begin{equation}
  \begin{aligned}
    & \frac{1}{2}\frac{d }{dt}\E(\by)+\nu \|\nabla \by\|^2 =
    \mathcal{E}_{m}(\bu)-\mathrm{e}_m(\bu)+(\bff,\by), 
    \\
    & \frac{1}{2}\frac{d }{dt}\E(\bz)+\nu \|\nabla \bz\|^2 \leq
    \mathcal{E}_{m}(\bu)+\mathrm{e}_m(\bu). 
    \label{eqn:energy-equation-pod-concise}
  \end{aligned}
\end{equation}

\begin{remark}
  We recall that apart from extremely simple geometries and provided one is willing to use
  in a systematic way special functions as the Bessel ones or the spherical harmonics
  (which are nevertheless time consuming in their evaluation), the explicit calculations
  in numerical tests will not be so easy to be obtained in a precise and efficient
  way. Hence, the solution of~\eqref{eq:system} and the long-time integration of its
  solutions poses hard numerical problems.
\end{remark}

We observe that, since the viscosity $\nu>0$ can be considered small in applications to
real-life flows, and since the interaction between the more energetic and less energetic
basis functions is weak, we can suppose that the integral
\begin{equation*}
  \nu\int_\Omega\nabla \by:\nabla\bz\,dx,
\end{equation*}
even if non vanishing is negligible in the long-time (at least in a first approximation).
We will show later on that --at least on the numerical tests-- this assumption is
reasonable and that the results obtained by using this approximation are numerically
sound.

We point out for the reader that it we have a first fundamental difference with respect to
the classical splitting based on the use of a spectral basis (which will be recalled in
Section~\ref{sec:spectral}, where the latter integral vanishes exactly. For this reason
In the next section we will show derivation of the corresponding system of equations,
which holds, when the eigenfunctions are used.
\subsection{On the spectral decomposition}
\label{sec:spectral}
In this section we compare the results of the previous section with the well-established
ones which can be proved if the spectral decomposition, i.e., that made with
eigenfunctions of the Stokes operator $\{\mathcal{W}_k\}$ is used, instead that of a
generic POD basis.  We recall that, by classical results about compact operators in
Hilbert spaces there exists a sequence of smooth functions $\{\mathcal{W}_{k}\}$ (and
their regularity is depending on the smoothness of the bounded domain $\Omega$) and an
increasing sequence of positive numbers $\{\lambda_{k}\}$ such that
\begin{equation*}
  A\mathcal{W}_k=\lambda_k\mathcal{W}_k\qquad \text{and}\qquad 
  \int_\Omega\mathcal{W}_k\cdot \mathcal{W}_j\,dx=\delta_{k j}.
\end{equation*}
Since each one the functions $\mathcal{W}_k$ solves the following Stokes system
$A\mathcal{W}_k=\lambda_k \mathcal{W}_k$, then it follows by an integration by parts that
\begin{equation*}
  \int_\Omega\nabla\mathcal{W}_k:
  \nabla \mathcal{W}_j\,dx=0\qquad\text{for }k\not=j,
\end{equation*}
hence also the $V$-orthogonality of the family $\{\mathcal{W}_k\}_{k\in\N}$.

We consider now the usual decomposition by eigenfunctions associated with low and high
frequencies
\begin{equation*}
  \bu=\by+\bz:=\sum_{k=1}^m c_k \mathcal{W}_k+\sum_{k=m+1}^\infty\  c_k
  \mathcal{W}_k=\mathbf{P}_m\,\bu+\mathbf{Q}_m\,\bu, 
\end{equation*}
where $\mathbf{P}_m$ is the projection over the subspace generated by
$\{\mathcal{W}_k\}_{k=1,\dots,m}$. Our main result is based on a standard result about the
projector $\mathbf{P}_m$, that can be found in~\cite[Appendix~A.4,
Theorem~4.11]{MNRR1996}:
\begin{theorem} 
  \label{thm:thm_proj} 
  The projector $\mathbf{P}_m$ can be defined as a continuous endomorphism over $V$, $H$
  and $V'$, and one has
  \begin{equation*}
    \| \mathbf{P}_m \|_{ \mathcal{L} (V,V)} \le 1, \qquad \| \mathbf{P}_m \|_{ \mathcal{L} (H,H)} \le
    1, \qquad \|  \mathbf{P}_m \|_{ \mathcal{L} (V',V')} \le 1. 
  \end{equation*}
\end{theorem} 
The result is mainly based on the regularity of solutions of elliptic equations, and
thanks to this fact, it is possible to decompose the equations for the velocity, which
yields,
\begin{equation*}
  \begin{aligned}
    & -\nu\int_\Omega \mathbf{P}_m (\Delta\bu)\cdot \by\,dx=-\nu\int_\Omega \Delta\by\cdot
    \by\,dx
    = \nu\,\|\nabla\by\|^2,
    \\
    &-\nu\int_\Omega \mathbf{Q}_m\, (\Delta\bu)\cdot \bz\,dx=-\nu\int_\Omega
    \Delta\bz\cdot \bz\,dx=
    \nu\,\|\nabla\bz\|^2,
  \end{aligned}
\end{equation*}
since $\mathbf{P}_m\Delta\bu=\Delta \mathbf{P}_m\,\bu=\Delta\by$ and also
$\mathbf{Q}_m (\Delta\bu)=\Delta\mathbf{Q}_m\, \bu=\Delta\bz$.

Thus, we directly obtain the system
\begin{equation}
  \label{eq:splitting-spectral}
  \begin{aligned}
    &\frac{1}{2}\frac{d }{dt}\|\by\|^2+\nu \|\nabla
    \by\|^2+(B(\by+\bz,\by+\bz),\by)=(\bff,\by),
    \\
    & \frac{1}{2}\frac{d }{dt}\|\bz\|^2+\nu \|\nabla
    \bz\|^2+(B(\by+\bz,\by+\bz),\bz)\leq\mathbf{0},
  \end{aligned}
\end{equation}
which is more tractable than~\eqref{eq:splitting_POD} from many points of view.  In
particular, the equations in~\eqref{eq:splitting-spectral}, which can be rewritten also as 
\begin{equation*}
  \begin{aligned}
    & \frac{1}{2}\frac{d }{dt}\E(\by)+\nu \|\nabla \by\|^2 =
    -\mathrm{e}_m(\bu)+(\bff,\by),
    \\
    & \frac{1}{2}\frac{d }{dt}\E(\bz)+\nu \|\nabla \bz\|^2 \leq \mathrm{e}_m(\bu),
  \end{aligned}
\end{equation*}
do not contain the term $\mathcal{E}_{m}(\bu)$.
%
%
%
%
%
%
%
%
%
\section{Preliminaries on long-time averages}
	\label{sec:time-averages}
Since we consider long-time averages for the NSE, we must consider solutions which
are global-in-time (defined for all positive times). Due to the well-known open problems
related to the NSE, this forces us to restrict ourselves to Leray-Hopf weak solutions~\cite{CF1988, JLL69}. By
using a then natural setting we take the initial datum $\bu(0)$ in $H$.
The classical Leray-Hopf result of existence (but not uniqueness) of a global weak
solution $\bu$ to the NSE holds when $\bff\in V'$, and the velocity $\bu$ satisfies
\begin{equation*}
  \bu \in L^2(\R_+; V) \cap L^\infty(\R_+; H).
\end{equation*}
Notice that consider in this paper the case where $\bff$ is time-independent, for
simplicity. However, the following results can be extended to the case where $\bff =
\bff(t) $ is time dependent, for $\bff$ belonging to a suitable class (see~\cite{BL2018}).

In order to properly set what we mean by long-time-averaging, let $\psi:\
\R^+\times\Omega\to\R^N$ be any tensor field related to a given turbulent flow ($N$ being
its order). The time-average over a time interval $[0,t]$ is defined by
\begin{equation}
  \label{eq:time-filtering}
  M_{t}(\psi )(x) :=\frac{1}{t}\int_0^t \psi (s,x)\,ds\qquad \text{for }t>0. 
\end{equation}
According to the standard turbulence modeling process, we then apply the averaging
operator $M_t$ to NSE~\eqref{eq:NSE} and also to~\eqref{eq:system}, to study the limits
when $t\to+\infty$. We recall that the long-times averages represent one of the few
observable and computable quantities associated to a highly variable turbulent flow.  We
will adopt the following standard notation for the long-time average of any field $\psi$
\begin{equation*}
  \mean{\psi}(x) := \lim_{t\to+\infty} M_t( \psi) (x),
\end{equation*} 
whenever the limit exists. (Without too much restrictions we can suppose that the limits
we write do exist, at least after extracting sub-sequences leaving the mathematical
difficulties, which can be treated with generalized Banach limits, for a more general and
abstract framework). Within this theory we can decompose the velocity as follows
\begin{equation*}
  \bu=\overline{\bu}+\bu',
\end{equation*}
where $\bu'$ represent the so called turbulent fluctuations.
We recall that time-averaging has been introduced by O.~Reynolds~\cite{Rey1895}, at least
for large values of $t$, and the ideas have been widely developed by
L.~Prandtl~\cite{Pra1925} in the case of turbulent channel flows.  The same ideas have
been also later considered in the case of fully developed homogeneous and isotropic
turbulence, such as grid-generated turbulence. In this case the velocity field is
postulated as oscillating around a mean smoother steady state, see for instance
G.-K.~Batchelor~\cite{Bat1953}. For further details on the role of time averaging in
turbulence, after the work of Stokes and Reynolds, we can recall a few recent papers and
books~\cite{BIL2006,BL2018,CL2014,Foi197273, JL2016,Lay2014,Lew2015}, where aspects of computation
and modeling are studied.  

We now observe that, by taking the time-averages of the NSE we have the following
estimates, see~\cite[Prop.~2.1]{Lew2015}
\begin{equation}
  \label{eq:long-time-estimates}  
  \begin{aligned}
    \|\bu(t)\|^2&\leq \|\bu(0)\|^2\,\e^{-\nu\, C_P t}+\frac{\|\bff\|^2}{\nu^2\,
      C_{P}}\big(1-\e^{-\nu\,C_P t}\big),\qquad \forall\,t>0
    \\
    \frac{1}{t}\int_0^t\|\nabla
    \bu(s)\|^2\,ds&\leq\frac{\|\bff\|^2}{\nu^2}+\frac{\|\bu(0)\|^2}{\nu t},\qquad
    \forall\,t>0, 
  \end{aligned}
\end{equation}
where $C_P$ is the best constant in the Poincar\'e inequality 
\begin{equation*}
  C_P\|\bu\|^2\leq\|\nabla \bu\|^2\qquad \forall\,\bu\in H^1_0(\Omega).
\end{equation*}
The above inequalities show that both $\|\bu(t)\|^2$ and
$\frac{1}{t}\int_0^t\|\nabla\bu(s)\|^2\,ds$ are uniformly bounded, hence we have the
following result
\begin{theorem}[cf.~\cite{BL2018,Lew2015}]
  \label{thm:main_theorem}
  Let $\bu(0) \in H$, $\bff\in V'$, and let $\bu$ be a
  global-in-time weak solution to the NSE~\eqref{eq:NSE}.  Then, there exist
\begin{enumerate}
\item a sequence $\{t_n\}_{n\in\N}$ such that $\lim\limits_{n \to \infty }t_n=+\infty$;
\item a vector field $\mean{\bu}\in V$;
\item a vector field $\bB\in L^{3/2}(\Omega)^3$;
\item a second order tensor field $\reyn \in L^3(\Omega)^9$;
\end{enumerate} 
 such that it holds:
\begin{enumerate}
\item[i)] When $n \to \infty$,
    \begin{equation*}
      \begin{aligned}
        &M_{t_n}(\bu)\rightharpoonup
        \mean{\bu}\qquad\text{weakly in }V,
        \\
        & M_{t_{n}}\big ((\bu\cdot\nabla)\,\bu\big )\rightharpoonup\bB
        \qquad \text{weakly in } L^{3/2}(\Omega)^3,
        \\
        &M_{t_n}(\bu{'} \otimes \bu{'} )\rightharpoonup
        \reyn\qquad\text{weakly in }L^3(\Omega)^9;
      \end{aligned}
    \end{equation*}
  \item[ii)] The Reynolds averaged equations:
\begin{equation}
  \label{eq:R}
  \left\{
    \begin{aligned}
     (\mean{\bu}\cdot\nabla)\, \mean{\bu} -\nu\Delta \mean{\bu}+\nabla
      \mean{p}+\nabla\cdot\reyn&=\mean{\bff}\qquad &\text{ in }\Omega,
      \\
      \nabla\cdot \mean{\bu}&=0\qquad &\text{ in }\Omega,
      \\
      \mean{\bu}&=\mathbf{0}\qquad &\text{ on }\Gamma,
    \end{aligned}  
  \right.
\end{equation}
 hold true in the weak sense;
\item[iii)] The equality $\bF = \bB - (\mean{\bu} \cdot \nabla )\, \mean{\bu} = \nabla
  \cdot \reyn$ is valid in $\mathcal{D}'(\Omega)$;
  \item[iv)] The following energy balance (equality) holds true
  \begin{equation*}
    \nu\,\|\nabla\mean{\bu}\|^2+(\nabla \cdot \reyn, \mean{\bu} ) = \,
    <\mean{\bff},\mean{\bu} >; 
  \end{equation*}
\item[v)] The tensor $\reyn$ is dissipative in average or, more precisely, the following
  inequality
\begin{equation}
  \label{eq:reyn_dissipative} 
  \epsilon:=\nu\,\mean{\|\nabla{\bu}'\|^2}\leq\int_\Omega(\nabla\cdot\reyn)\cdot\mean{\bu}\,dx,
\end{equation}
holds true.
\end{enumerate}
\end{theorem}
It is important to observe that the long-time limit is characterized by the solution of
the system~\eqref{eq:R}, which is similar to the Navier-Stokes equations, but which
contains an extra term, coming from the effect of fluctuations, which has the mean effect
of increasing the dissipation.

We observe that this is related to the long-time behavior of solutions close to statistical
equilibrium. The study of the long-time behavior dates back to pioneering works of Foias
and Prodi on deterministic statistical solutions, see for instance~\cite{FMRT2001a}. Their
interest is mainly devoted to finding measure in the space of initial data to be connected
with the long-time limits.  Here, we follow a slightly different path, as
in~\cite{BL2018,Lew2015}, in order to characterize in a less technical way the long-time
behavior, without resorting to any ergodic-type result and also with the perspective that
long time averages are computable or at least can be approximated in a clear way.
\section{Average transfer of energy at equilibrium}
	\label{sec:average-energy}
Our goal is now to characterize in some sense the energy transfer between the two
functions $\by$, $\bz$ of the expansion and to determine --if possible-- the sign of
$\mathrm{e}_m(\bu)$, at least in an average sense. 

The point concerning the exchange of energy between low and high modes is in the same
spirit as the results recalled in Foias, Manley, Rosa, and Temam~\cite[Chap.~5]{FMRT2001a}
and follows from results obtained in a more heuristic way by Kolmogorov~\cite{Kol1941}.

We first observe that the $L^2$-orthogonality of the POD decomposition implies that
\begin{equation*}
  \|\bu\|^2=\|\by+\bz\|^2=\|\by\|^2+\|\bz\|^2.
\end{equation*} 
Hence, from the uniform $L^2$-bound on $\bu$ it follows that both $\by$ and $\bz$ are
uniformly bounded in time. From this observation we can deduce the following result,
reminding that $ \mathrm{e}_m$ and and $\mathcal{E}_{m}$ are defined by
equations~\eqref{eq:em} and~\eqref{eq:Em}, and $M_t$ is defined by
equation~\eqref{eq:time-filtering}.
\begin{theorem}
  \label{thm:thm-ROM}
  There exist a sequence $\{t_n\}$ such that $t_n\to+\infty$ and a field $\mean{\bz}\in H$
  such that
  \begin{equation}
    \label{eq:weak_cv_z_t}
    {\bf Z}_{t_n} = M_{t_n}(\bz)\rightharpoonup
    \mean{\bz}\qquad\text{weakly in }H,
      \end{equation}
      and 
      \begin{equation} 
        \label{eq:energy_sign}
        \liminf_{n\to+\infty}M_{t_{n}}( \mathrm{e}_m(\bu)+\mathcal{E}_{m}(\bu)) \geq0.
      \end{equation}
\end{theorem}
\begin{proof} 
  Let us observe first that by the energy inequality~\eqref{eq:long-time-estimates}, we
  easily deduce that $(M_t (\bz))_{t>0}$ is bounded in $H$, hence the first assertion of
  the statement and~\eqref{eq:weak_cv_z_t}. We next prove~\eqref{eq:energy_sign}.  To do
  so, we average with respect to time with the operator $M_t$ the balance
  equation~\eqref{eqn:energy-equation-pod-concise} for $\E(\bz)$, which yields 
  \begin{equation}
    \label{eq:mean_ener}
    \frac{1}{2t}\|\bz(t)\|^2-\frac{1}{2t}\|\bz(0)\|^2+\nu M_t (\|\nabla\bz\|^2)\leq
    M_t( \mathrm{e}_m(\bu)+\mathcal{E}_{m}(\bu)). 
  \end{equation}
  By using the energy inequality~\eqref{eq:long-time-estimates} once again, we see that
  the first two terms vanish as $t\to+\infty$ and also that $M_t (\|\nabla \bz\|^2)$ is
  bounded. Therefore,~\eqref{eq:mean_ener} yields
  \begin{equation*}
  0 \le   \nu \liminf_{n \to +\infty}  M_{t_n} (\|\nabla
    \bz\|^2) \le  \liminf_{n\to+\infty}M_{t_n} \mathrm{e}_m(\bu)+\mathcal{E}_{m}(\bu)),
  \end{equation*}
  hence~\eqref{eq:energy_sign}. We observe that in this case we do not have any direct
  estimation on the behavior of the $H^1$-norm.
\end{proof}
In the case we can assume that the limit exists, we also have the following result.
\begin{corollary}
  Let us assume the limit of $M_{t_n}(\mathrm{e}_m(\bu))$ for $n\to+\infty$ exists, and
  that 
  \begin{equation*}
    \liminf_{T\to+\infty}\frac{\nu}{T}\int_0^T(\nabla\bz(s),\nabla\by(s))\,ds = \liminf_{t
      \to \infty} \mathcal{E}_{m}(\bu) \geq0. 
  \end{equation*}
Then, it follows 
  \begin{equation*}
    \overline{\mathrm{e}_m(\bu)}=\lim_{n\to+\infty}\frac{1}{t_n}\int_0^{t_n}\e_m(\bu(s))\,ds\geq0.  
  \end{equation*}
\end{corollary}
This result can be interpreted as that, beyond the range of injection of energy, the
average net transfer of energy occurs only into the small scales.  This occurs if the term
of interaction between gradients of large and small scales is negligible, in the limit of
long times. This latter assumption is not proved rigorously, but we will see it is
satisfied in the numerical tests, with a good degree of approximation (see
Section~\ref{sec:numerical-results}). However, when one uses the eigen-vectors of the
Stokes operator as POD basis, this is automatically satisfied since this basis is also
orthogonal for the $H^1$-scalar product, so that in this case $ \mathcal{E}_{m}(\bu)= 0$. 
\subsection{The spectral case}
The results of the previous section can be made much more precise in the case of
decomposition made by a spectral basis of eigenfunctions of the Stokes operator.  We
present the results, which are in some sense new and not fully completely included
in~\cite{FMRT2001a}, in the sense of time-averaging. This procedure is applied to $\bu$,
which is a weak solution of the Navier-Stokes equations, satisfying the uniform
estimates~\eqref{eq:long-time-estimates}. In this way, the orthogonality (in both $H$ and
$V$) of the basis implies that
\begin{equation*}
  \|\bu\|^2=\|\by\|^2+\|\bz\|^2\qquad\text{and}\qquad   \|\nabla \bu\|^2=\|\nabla
  \by\|^2+\|\nabla \bz\|^2. 
\end{equation*}

The results in this case are more precise than those from Theorem~\ref{thm:thm-ROM}, since
we have at disposal a larger set of a priori estimates and also the set of
equations~\eqref{eq:splitting-spectral} has a better structure than~\eqref{eq:splitting_POD}.

We now prove the following results in the case of a decomposition of the velocity into
small and large frequencies. The first one aims at taking the time average and then let
$t$ go to infinity in the equations~\eqref{eq:splitting_POD} satisfied by $\by$ and
$\bz$. The second one aims at comparing the amount of turbulent dissipation of small and
large frequencies with respect to the total work of the corresponding Reynolds stresses
$\reyny$ and $\reynz$.
\begin{theorem}
  \label{thm:newtheorem}
  Let $\bu(0) \in H$, $\bff\in \mathbf{P}_m\, H$, and let $\bu$ be a global-in-time
  weak solution to the NSE~\eqref{eq:NSE}.  Then, there exist
  \begin{enumerate}
  \item a sequence $\{t_n\}_{n\in\N}$ such that $\lim\limits_{n \to \infty }t_n=+\infty$;
  \item vector fields $\mean{\by},\mean{\bz}\in V$;
  \item  vector fields $\bB_1,\bB_2\in V'$;
  \end{enumerate} 
 such that it holds:
 \begin{enumerate}
 \item[i)] When $n \to \infty$,
   \begin{equation*}
     \begin{aligned}
       \begin{aligned}
         &M_{t_n}(\by)\rightharpoonup \mean{\by}\qquad\text{weakly in }V,
         \\
         & M_{t_{n}}\big ((\by\cdot\nabla)\,\by\big )\rightharpoonup\bB_1 \qquad \text{weakly in }V',
       \end{aligned}&\hspace{1.5cm} \begin{aligned} &M_{t_n}(\bz)\rightharpoonup
         \mean{\bz}\qquad\text{weakly in }V,
         \\
         &M_{t_{n}}\big
         ((\bz\cdot\nabla)\,\bz\big )\rightharpoonup\bB_2 \qquad 
         \text{weakly in } V',
       \end{aligned}
     \end{aligned}
   \end{equation*}
 \item[ii)] The Reynolds averaged equations:
   \begin{equation}
     \label{eq:Ry}
     \left\{
       \begin{aligned}
         -\nu\Delta \mean{\by}+\nabla \mean{p_\by}+\bB_1&=\mathbf{P}_m\bff\qquad &\text{ in
         }\Omega,
         \\
         \nabla\cdot \mean{\by}&=0\qquad &\text{ in }\Omega,
      \\
      \mean{\by}&=\mathbf{0}\qquad &\text{ on }\Gamma,
    \end{aligned}  
  \right.
\end{equation}
and 
\begin{equation}
  \label{eq:Rz}
  \left\{
    \begin{aligned}
      -\nu\Delta \mean{\bz}+\nabla \mean{p_\bz}+\bB_2&=\mathbf{0}\qquad &\text{ in
      }\Omega,
      \\
      \nabla\cdot \mean{\bz}&=0\qquad &\text{ in }\Omega,
      \\
      \mean{\bz}&=\mathbf{0}\qquad &\text{ on }\Gamma,
    \end{aligned}  
  \right.
\end{equation}
holds true in $V'$.
\end{enumerate}
\end{theorem} 
Arguing as in~\cite{BL2018, Lew2015}, using the relations
$(\mean{\bz}\cdot\nabla)\,\mean{\bz}=\nabla\cdot (\mean{\bz} \otimes \mean{\bz})$ and
$(\mean{\by}\cdot\nabla)\,\mean{\by} = \nabla\cdot (\mean{\by} \otimes \mean{\by})$, we
get the existence of ``small frequencies'' and ``large frequencies'' Reynolds stresses
$\reyny$ and $\reynz$ in $V'$, such that
\begin{equation*} 
  B_1=  \nabla \cdot \reyny + (\mean{\by} \cdot \nabla)\,\mean{\by}\qquad\text{and}\qquad
  B_2=\nabla\cdot\reynz+(\mean{\bz} \cdot \nabla)\,\mean{\bz},  
\end{equation*}
or equivalently, if we write the Reynolds decomposition as 
\begin{equation*} 
  \by = \mean{\by} + \by' \quad\text{and} \quad \bz = \mean{\bz} + \bz',
\end{equation*}
then 
\begin{equation*} 
  \reyny= 
  \overline { \by' \otimes \by'}\qquad\text{and}
  \qquad \reynz= 
  \overline { \bz' \otimes \bz'}, 
\end{equation*}
where the bar operator denotes the limit of the $M_{t_n}$'s in $V'$ as $n \to \infty$
(eventually after having extracted another sub-sequence). 

According to the budget~\eqref{eq:reyn_dissipative}, we aim to compare the turbulent
dissipation of small and large scales, denoted as $\varepsilon^\downarrow$ and
$\varepsilon^\uparrow$ respectively, to the total work of the Reynolds stresses, namely
$(\nabla \cdot\reyny, \mean{\by})$ and $(\nabla \cdot\reynz, \mean{\bz})$, where
\begin{equation*}  
  \epsilon^\downarrow:=\nu\,\mean{\|\nabla{\by}'\|^2}\quad\text{and} \quad
  \epsilon^\uparrow:=\nu\,\mean{\|\nabla{\bz}'\|^2}.
\end{equation*}
When we compare to~\eqref{eq:reyn_dissipative}, we observe that triad nonlinear effect
between small and large frequencies will be felt, that means the nonlinear interactions
due to the convection will be provided by the term
  \begin{equation}
    \label{eq:def-phi}
    \Phi_{\bz}(\by):=(\mathbf{Q}_m [(\by+\bz)\cdot\nabla(\by+\bz)],\bz)
    =-\Phi_{\by}(\bz)=-(\mathbf{P}_m [(\by+\bz)\cdot\nabla(\by+\bz)],\by).
  \end{equation}
%
  Notice that due to the regularity of $\by$, it is easily checked that the following
  energy balance holds true (this property will be shown with more details in the proof of
  Theorem~\ref{thm:newtheorem} below)
  \begin{equation*}
    \nu\,\|\nabla\mean{\by}\|^2+ 
   ( \nabla \cdot \reyny, \mean{\by})= \, <{\bff},\mean{\by} >.
  \end{equation*}
  Finally, we will use the following orthogonality relation (see
  e.g.~\cite[Lemma~4.4]{BL2018}), formally written as follows
\begin{equation} 
  \label{eq:orthogonality} 
  \overline {\| \nabla \psi \|^2} = \| \nabla \mean{\psi } \|^2 + \overline { \| \nabla \psi ' \|^2 }. 
\end{equation}
  %
\begin{theorem}  
  \label{thm:newtheorem2} 
  The families $(M_t( \Phi_{\bz}(\by)))_{t>0}$ and $(M_t( \Phi_{\by}(\bz)))_{t>0}$ 
  converge (along certain subsequences) as $t \to \infty$. Let $\mean{\Phi_{\bz}(\by)}$ and
  $\mean{\Phi_{\by}(\bz)}$ denote the corresponding limits. One has
  \begin{equation} 
    \label{eq:sign_psi} 
    \mean{\Phi_{\bz}(\by)} = - \mean{\Phi_{\by}(\bz)}\le 0, 
  \end{equation} and the following dissipation balances hold true
  \begin{eqnarray} 
    && \label{eq:dissipation}
    \epsilon^\downarrow + \mean{\Phi_{\by}(\bz)} = ( \nabla
    \cdot \reyny, \mean{\by}), 
    \\
    &&  \label{eq:dissipation2}
    \epsilon^\uparrow + \mean{\Phi_{\bz}(\by)} \le ( \nabla \cdot \reynz, \mean{\bz}).
  \end{eqnarray}
\end{theorem}
\begin{remark} 
  Notice that by equations~\eqref{eq:sign_psi} and~\eqref{eq:dissipation} we
  see that $ \reyny$ is dissipative in mean, and follows the same
  law~\eqref{eq:reyn_dissipative} as the complete Reynolds stress, namely
  \begin{equation*}
    \epsilon^\downarrow \le ( \nabla \cdot \reyny, \mean{\by}).  
  \end{equation*}
  However, nothing similar can be concluded from~\eqref{eq:dissipation} about $\reynz$,
  that might be at this stage non dissipative in mean, which permits an inverse energy
  cascade to occur.
\end{remark}
 
The results of Theorems~\ref{thm:newtheorem} and~\ref{thm:newtheorem2} are original, even
if similar results have been already obtained in~\cite{Foi197273} and reported also
in~\cite{FMRT2001a}. In that case, the results are based on the notion of deterministic
statistical solutions and on a sort of ergodic hypothesis. Even if statements could look
very similar to ours, the main difference is that we do not average over the set $H$ of
initial data, and we do not introduce probability measures on $H$, as suggested by the
work by Prodi~\cite{Pro1960,Pro1961}. Our approach is based on a more elementary
functional setting and also amenable to include treatment of sets of external forces, as
those in several numerical or practical experiments. The main point is an extension of the
results in~\cite{BL2018}.

\begin{proof}[Proof of Theorem~\ref{thm:newtheorem}]
  We know, from the results in~\cite{BL2018,Lew2015} that $\mathbf{U}_t=M_t (\bu)$ is such
  that
  \begin{equation*}
    \begin{aligned}
      & \mathbf{U}_t\rightharpoonup \overline{\bu}\qquad \text{weakly in }V,
      \\
      &M_t ((\bu\cdot\nabla)\, \bu)\rightharpoonup \mathbf{B}\qquad \text{ in
      }L^{3/2}(\Omega)\subset V' ,
    \end{aligned}
  \end{equation*}
  hence, if we define $\mathbf{F}:=\mathbf{B}-(\overline{\bu}\cdot\nabla)\,\overline{\bu}$,
  we get
  \begin{equation*}
   \nu\, (\nabla \overline{\bu},\nabla
    {\pmb\phi})+\big((\overline{\bu}\cdot\nabla)\,\overline{\bu},{\pmb\phi}\big)
    +<\mathbf{F},{\pmb\phi}>\,=\,<\bff,{\pmb\phi}>,     
  \end{equation*}
  and using $\overline{\bu}\in V$ as test function we obtain
  \begin{equation*}
    \nu\,\|\nabla \overline{\bu}\|^2+<\mathbf{F},\overline{\bu}>\,=\,<\bff,\overline{\bu}>.
  \end{equation*}

  \medskip

  We assume now that $\mathbf{P}_m \bff=\bff$, and we consider the equations
  satisfied by $\by=\mathbf{P}_m\bu$ and $\bz=\mathbf{Q_m}\bu$. In particular, the
  equation for $\by$ reads, as an abstract equation in $V_m=\mathbf{P}_m V$, as follows:
  \begin{equation*}
    \frac{d\by}{dt} +\nu A\by+\mathbf{P}_m[(\by+\bz)\cdot\nabla(\by+\bz)]=\mathbf{P}_m\bff.
  \end{equation*}
  The uniform estimates on $\bu$ from Theorem~\ref{thm:main_theorem} combined with
  Theorem~\ref{thm:thm_proj}, about the properties of the projection operator
  $\mathbf{P}_m$ as a continuous operator over $V'$, yield
  \begin{equation*}
    \begin{aligned}
      & M_{t}(\by)=\mathbf{Y}_t\rightharpoonup \overline{\by} \qquad\text{weakly in }V,
      \\
      &\mathbf{P}_m M_t ((\bu\cdot\nabla)\, \bu)=\mathbf{P}_m M_t
      [(\by+\bz)\cdot\nabla(\by+\bz)]\rightharpoonup \mathbf{P}_m
      \mathbf{B}=\mathbf{B}_1\qquad \text{weakly in }V',
    \end{aligned}
  \end{equation*}
  in such a way that $\mean{\by}$ satisfies
  \begin{equation*}
    \nu\,(\nabla\overline{\by},\nabla\mathcal{W}_k)+<(\overline{\by}\cdot\nabla)\,\overline{\by},\mathcal{W}_k> 
    +<\mathbf{F}_{\by},\mathcal{W}_k>\,=\,<\bff,\mathcal{W}_k>\qquad  \text{for all}1\leq k\leq   m,  
  \end{equation*}
  where $\mathbf{F}_{\by}:=\mathbf{B}_1-(\overline{\by}\cdot\nabla)\,\overline{\by}$,
  which leads to~\eqref{eq:Ry} by De Rham Theorem. Arguing as in~\cite{BL2018} (which was
  already mentioned above), it is easily checked that there exists $\reyny$ such that
  $\mathbf{F}_{\by} = \nabla \cdot \reyny$. Hence, being $\overline{\by}\in V$ a
  legitimate test function, we get
  \begin{equation}
    \label{eq:first-equality-y}  
    \nu\,\|\nabla\overline{\by}\|^2+<\mathbf{F}_{\by},\overline{\by}>\,=\,
    \nu\,\|\nabla\mean{\by}\|^2+ ( \nabla \cdot \reyny, \mean{\by})= \, <\bff,\overline{\by}>. 
  \end{equation}

\medskip

  The other term $\bz$ of the decomposition satisfies
  \begin{equation*}
    \frac{d}{dt} \bz+\nu A\bz+\mathbf{Q}_m[(\by+\bz)\cdot\nabla(\by+\bz)]=0.
  \end{equation*}
  The uniform estimates on $\bu$ and the boundedness of $\mathbf{P}_m$ imply the following
  convergence (up to a sub-sequence), as already shown in Theorem~\ref{thm:thm-ROM}
  \begin{equation*}
    \begin{aligned}
      &M_{t}(\bz)= \mathbf{Z}_t\rightharpoonup \overline{\bz} \qquad \text{weakly in }V,
      \\
      &\mathbf{Q}_m M_t ((\bu\cdot\nabla)\, \bu)=\mathbf{Q}_m M_t
      [(\by+\bz)\cdot\nabla(\by+\bz)]\rightharpoonup 
      \mathbf{Q}_m\mathbf{B}=\mathbf{B}_2 \qquad \text{weakly in }V'.
    \end{aligned}
  \end{equation*}
  By using that $\mathbf{B}=\mathbf{P}_m \mathbf{B}+\mathbf{Q}_m\,\mathbf{B}$, 
  we get
  \begin{equation*}
    \nu\,(\nabla\overline{\bz},\nabla\mathcal{W}_j)+<(\overline{\bz}\cdot\nabla)\,\overline{\bz},\mathcal{W}_j>+
    <\mathbf{F}_{\bz},\mathcal{W}_j>\,=0\qquad \text{for all } j\geq m+1,   
  \end{equation*}
  for $\mathbf{F}_{\bz}=\mathbf{B}_2-(\overline{\bz}\cdot\nabla)\,\overline{\bz} = \nabla
  \cdot \reynz$. Hence,~\eqref{eq:Rz} follows again by De Rham Theorem.  Notice that, by
  using as test function $\mean{\bz}\in V$ we have the following energy equality:
  \begin{equation}
    \label{eq:star}
   \nu\, \|\nabla \overline{\bz}\|^2+(\nabla \cdot \reynz, \overline{\bz}) \,=0.
  \end{equation}
  which concludes the proof. 
\end{proof}
\begin{proof}[Proof of Theorem~\ref{thm:newtheorem2}]  We now write the energy inequality
  for $\bz$, obtaining 
  \begin{equation*}
    \frac{1}{2}  \frac{d}{dt}\|\bz\|^2+\nu\,\|\nabla \bz\|^2+(\mathbf{Q}_m
    [(\by+\bz)\cdot\nabla(\by+\bz)],\bz)\leq0,
  \end{equation*}
  and hence, by using the orthogonality of the basis, we have that $\mathbf{Q}_m\bz=\bz$ and
  \begin{equation*}
    \frac{1}{2}\frac{d}{dt}\|\bz\|^2+\nu\,\|\nabla \bz\|^2+\big( [(\by+\bz)\cdot\nabla(\by+\bz)],\bz\big)=
    \frac{1}{2}  \frac{d}{dt}\|\bz\|^2+\|\nabla \bz\|^2+\Phi_{\bz}(\by)\leq0,
  \end{equation*}
  recalling the definition of $\Phi_{\bz}(\by)$ in~\eqref{eq:def-phi}.
  
  Averaging over a fixed time interval $(0,t)$, we get
  \begin{equation*}
    \frac{1}{2t}\|\bz(t)\|^2-
    \frac{1}{2t}\|\bz(0)\|^2+ \nu  M_t( \|\nabla\bz\|^2) +M_t (\Phi_{\bz}(\by)) \leq0.  
  \end{equation*}
  The $L^2$-uniform bounds on $\bz$ imply that $\frac{1}{2t}\|\bz(t)\|^2\to0$, hence,
  possibly after having extracted another sub-sequence to ensure the convergence of the
  term $M_t( \|\nabla\bz\|^2)$ (that is known to be bounded by the energy
  inequality~\eqref{eq:long-time-estimates}) we get
  \begin{equation} 
    \label{eq:energy_4} 
    \limsup_{n \to \infty}  M_{t_n} (\Phi_{\bz}(\by)) \le - \nu\,
    \mean{\|\nabla\bz\|^2}\le 0.  
  \end{equation}
  We now combine the orthogonality relation~\eqref{eq:orthogonality} with the energy
  balance~\eqref{eq:star}, so that~\eqref{eq:energy_4} yields
  \begin{equation*} 
    \epsilon^\uparrow+\limsup_{n \to \infty}  M_{t_n}(\Phi_{\bz}(\by)) \le ( \nabla
    \cdot\reynz, \mean{\bz}), 
  \end{equation*} 
  which is almost inequality~\eqref{eq:dissipation2}, up to the convergence of
  $(M_t(\Phi_{\bz}(\by)))_{t >0}$ that remains to be proved.  To prove this, we deal with
  the budget for $\by$, reminding that
  \begin{equation} 
    \label{eq:egality_psi} 
    \Phi_{\bz}(\by)=(Q_m [(\by+\bz)\cdot\nabla(\by+\bz)],\bz)=-(\mathbf{P}_m
    [(\by+\bz)\cdot\nabla(\by+\bz)],\by)=-\Phi_{\by}({\bz}). 
  \end{equation}
  Then, averaging the energy equality (in this case we have equality since $\by$ solves a
  finite dimensional system of ordinary differential equations) which is satisfied for
  $\by$, we get
  \begin{equation} 
    \label{eq:energy_5} 
    \frac{1}{2{t}}\|\by({t})\|^2-
    \frac{1}{2{t}}\|\by(0)\|^2+ \nu M_t(\|\nabla\by\|^2)+M_t(\Phi_{\by}(\bz))
    =\,<\bff,M_t(\by)>.
  \end{equation}
  By the same argument, eventually after having extracted a further sub-sequence,
  $(M_t(\|\nabla\by\|^2))_{t>0}$ is convergent as $t_n \to \infty$, as well as
  $(<\bff,M_t(\by)>)_{t >0}$. Therefore, $\{M_{t_n}(\Phi_{\by}(\bz))\}$ is also convergent
  by~\eqref{eq:energy_5}. Let $\mean{\Phi_{\by}(\bz)}$ denotes its limit. In particular,
  by~\eqref{eq:egality_psi}, $\{M_{t_n} (\Phi_{\bz}(\by)))\}$ is also convergent, with
  limit $\mean{\Phi_{\bz}(\by)} = - \mean{\Phi_{\by}(\bz)}$. We are done
  with~\eqref{eq:dissipation2}. 

  It remains to check~\eqref{eq:dissipation}.  Taking the limit as $n \to \infty$
  in~\eqref{eq:energy_5} gives the equality
  \begin{equation*}  
    \nu\,\overline{ \|\nabla\by\|^2 } + \mean{\Phi_{\bz}(\by)}= \,<\bff,\overline{\by}>, 
  \end{equation*}
  which, combined with the energy balance~\eqref{eq:first-equality-y} and the orthogonality
  relation~\eqref{eq:orthogonality}, yields~\eqref{eq:dissipation}, ending the proof.
\end{proof}
%
%
\section{Numerical results}
	\label{sec:numerical-results}
%
%
In Theorem~\ref{thm:thm-ROM}, we showed that 
      \begin{equation}
        \liminf_{T\to+\infty}\frac{1}{T}\int_0^{T} \mathrm{e}_m(\bu(s)) + \mathcal{E}_{m}(\bu(s)) \,ds \geq0 .
        \label{eqn:numerical-results-1}
      \end{equation}
In this section, we investigate numerically whether the inequality~\eqref{eqn:numerical-results-1} holds.
To this end, we consider the one-dimensional Burgers equation with homogeneous boundary conditions as a simplified, yet relevant test case:
\begin{equation}
  \begin{cases}
    \displaystyle u_t -\nu u_{xx} + u u_x  = f ~~~~~~~~~~~~ (x,t) \in \Omega\times [0,1],
    \\
    \displaystyle~~~~~~~~~~~~~~~~~~~~~~ u = 0 ~~~~~~~~~~ (x,t) \in \partial\Omega\times [0,1].
  \end{cases}
  \label{eqn:burgers}
\end{equation}

To calculate the long-time average of $\mathrm{e}_m(u)$ in~\eqref{eqn:numerical-results-1}, we use the composite trapezoidal rule: 
\begin{eqnarray}
	\frac{1}{T} \int_{0}^{T} \mathrm{e}_m(u(s)) \, ds 
	\approx \frac{1}{2n} \sum_{i=1}^{n} \bigl( \mathrm{e}_m(u(t_i))+\mathrm{e}_m(u(t_{i+1})) \bigr) ,
	\label{eqn:trapezoidal}
\end{eqnarray} 
where $t_i = (i-1)*\frac{T}{n},~i=1,...,n+1.$
We also use the composite trapezoidal rule to calculate the long-time average of $\mathcal{E}_{m}(u)$.

\subsection{Numerical Results with step function initial condition}
Our numerical results are obtained by using the one-dimensional Burgers equation~\eqref{eqn:burgers}
with a step function initial condition~\cite{iliescu2014are,xie2018data}:
\begin{equation}
  u_0(x)=\begin{cases}
    \displaystyle~ 1, & x \in (0,1/2],
    \\
    ~\displaystyle 0, & x \in (1/2,1]. 
    \label{step function}
\end{cases}
\end{equation}
We use the following parameters in the finite element discretization of the Burgers equation~\eqref{eqn:burgers}: $\Omega =  [0,1]$, $\nu=10^{-2}$, $f = 0$, mesh size $h=1/128$,  piecewise linear finite element spatial discretization, and backward Euler time discretization.

\subsubsection{Case 1:}

For this test case, we consider the time interval $[0,T] = [0,1]$ and the time step $\Delta t = 10^{-2}$. 
We utilize all the snapshots to build the POD basis, whose dimension is $d=37$.
In the composite trapezoidal rule, we use $n=100$.
In Figure~\ref{fig:numerical-results-1}, we plot the DNS results (which are used to generate the snapshots).
In Table~\ref{tab:numerical-results-1}, we list the time-averages of $\mathrm{e}_m(u)$ and $\mathcal{E}_{m}(u)$ for different $m$ values.
We note that the time-average of $\mathrm{e}_m(u)$ is positive for all $m$ values.
The time-average of $\mathcal{E}_{m}(u)$ is positive for the low $m$ values and negative for the largest $m$ values.
Furthermore, the magnitude of the time-average of $\mathcal{E}_m(u)$ is lower than the magnitude of the time-average of $\mathrm{e}_m(u)$.
Thus, we conclude that the time average $\frac{1}{T}\int_0^{T} \mathrm{e}_m(u(s)) + \mathcal{E}_{m}(u(s)) \,ds$ in~\eqref{eqn:numerical-results-1} is positive for all $m$ values.

\begin{figure}[h]\centering
  \includegraphics[width= 0.8\textwidth]{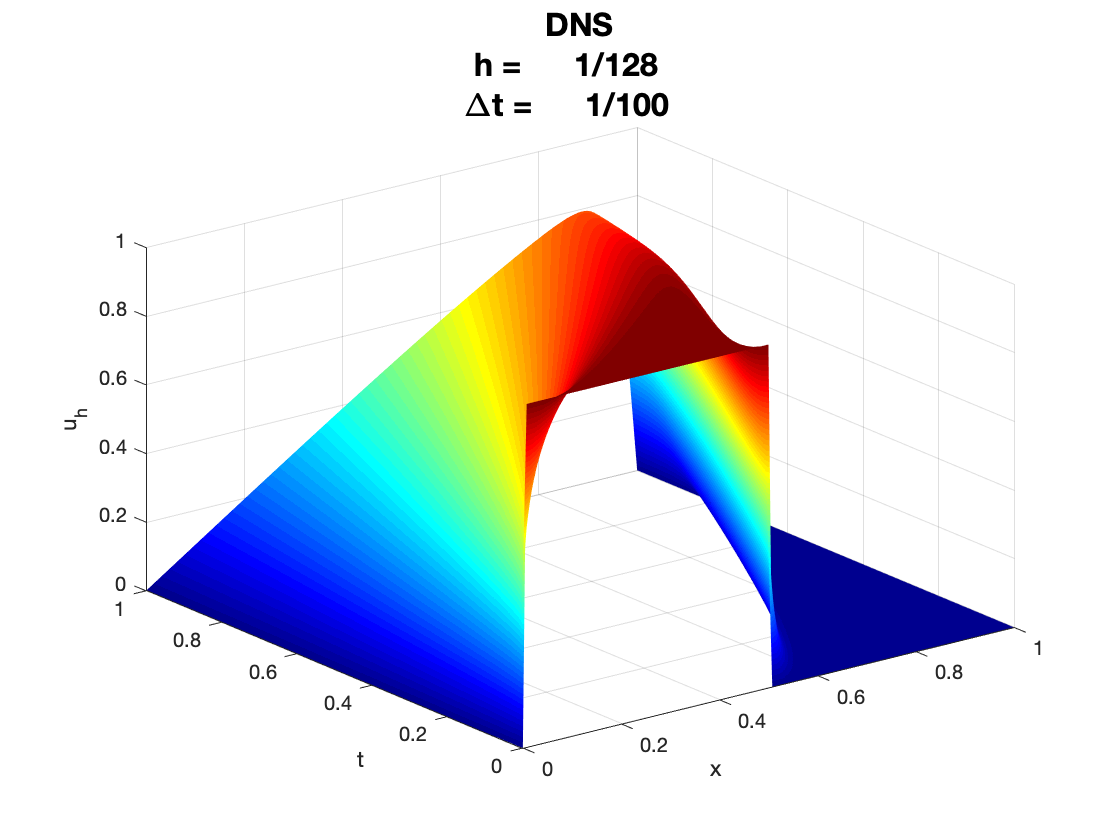}
 \label{fig:numerical-results-1}
\end{figure}

\begin{table}[h]
		\begin{tabular}{|c|c|c|c|}
			\hline
			\multicolumn{1}{ |c| }{  $m$  } &		
			\multicolumn{1}{ |c| }{ $\displaystyle  \int_{0}^{1} \mathrm{e}_m(u(s)) \, ds$ }  &
			\multicolumn{1}{ |c| }{ $\displaystyle  \int_{0}^{1} \mathcal{E}_m(u(s)) \, ds$ }  \\		    
			\hline
			3 &2.6170e-02 & 1.0451e-03\\ 
			\hline
			6 & 7.3208e-03 & 2.2524e-03  \\ 
			\hline
			9 &   1.4934e-03 & 1.5977e-03  \\ 
			\hline
			15 &   3.6181e-05 & 3.6287e-05 \\ 
			\hline
			20 &  1.2776e-06 & 7.0356e-07 \\ 
			\hline
			25 &  4.6207e-08  & 9.5638e-09 \\ 
			\hline
			30 &   2.0257e-09 & -1.1127e-10 \\ 
			\hline
			35 &   8.2806e-11 &  -2.2595e-11 \\ 
			\hline
			
		\end{tabular}
		\caption{
			Case 1: Time-averages of $\mathrm{e}_m(u)$ and $\mathcal{E}_{m}(u)$ for $d=37$ and different $m$ values.
			\label{tab:numerical-results-1}			
			}
\end{table}

In Case 1, we showed that the time average $\frac{1}{T}\int_0^{T} \mathrm{e}_m(u(s)) + \mathcal{E}_{m}(u(s)) \,ds$ in~\eqref{eqn:numerical-results-1} is positive.
In the remainder of this section, we investigate whether this time average remains positive if we make the following changes in our computational setting:
(i) we increase/decrease the time-interval; and 
(ii) we use more quadrature points (i.e., subintervals) in the composite trapezoidal rule~\eqref{eqn:trapezoidal}.


\subsubsection{Case 2:}
In this case, we use a longer time interval, i.e., $[0,T] = [0,10]$ (instead of $[0,T] = [0,1]$, as we used in Case 1). 
We also use different time step ($\Delta t$) values to generate the snapshots and a different number of quadrature points to evaluate the time average $\frac{1}{T}\int_0^{T} \mathrm{e}_m(u(s))$.

In Tables~\ref{tab:numerical-results-2}--\ref{tab:numerical-results-5}, we list the time-averages of $\mathrm{e}_m(u)$ and $\mathcal{E}_{m}(u)$ for different time steps ($\Delta t$) values, different number of equally spaced quadrature points ($n$), and different $m$ values.
We note that the time-averages of $\mathrm{e}_m(u)$ and $\mathcal{E}_{m}(u)$ are positive for all $\Delta t$, $n$, and $m$ values.
Furthermore, the magnitude of the time-average of $\mathcal{E}_m(u)$ is lower than the magnitude of the time-average of $\mathrm{e}_m(u)$.
Thus, we conclude that the time average $\frac{1}{T}\int_0^{T} \mathrm{e}_m(u(s)) + \mathcal{E}_{m}(u(s)) \,ds$ in~\eqref{eqn:numerical-results-1} is positive for all $\Delta t$, $n$, and $m$ values.
Furthermore, we note that decreasing the time step while keeping the same number of snapshots (i.e., $n=10000$) does not change the time average $\frac{1}{T}\int_0^{T} \mathrm{e}_m(u(s)) + \mathcal{E}_{m}(u(s)) \,ds$ significantly (see Tables~\ref{tab:numerical-results-3}--\ref{tab:numerical-results-5}).

\begin{table}[h]
		\begin{tabular}{|c|c|c|c|}
			\hline
			\multicolumn{1}{ |c| }{  $m$  } &	
			\multicolumn{1}{ |c| }{ $\displaystyle  \frac{1}{10} \int_{0}^{10} \mathrm{e}_m(u(s))\,ds$ }  &
			\multicolumn{1}{ |c| }{ $\displaystyle  \frac{1}{10} \int_{0}^{10} \mathcal{E}_m(u(s)) \,ds$ } \\		   
			\hline
			3 &   3.3652e-03 & 8.6523e-05 \\ 
			\hline
			6 &   1.0001e-03 & 2.1570e-04 \\ 
			\hline
			9 &  2.1132e-04 & 1.8614e-04  \\ 
			\hline
			15 &  5.4224e-06 & 5.5724e-06  \\ 
			\hline
			20 &   4.3119e-07 & 2.5667e-07  \\ 
			\hline
			25 &   5.6884e-08 &  1.0451e-08  \\ 
			\hline
			30 &   1.1327e-09 & 3.2736e-10  \\ 
			\hline
			35 &   2.0469e-11 & 2.6396e-12  \\ 
			\hline
			
		\end{tabular}
		\caption{
			Case 2: Time-averages of $\mathrm{e}_m(u)$ and $\mathcal{E}_{m}(u)$ for $d=38$, $\Delta t = 10^{-2}$, $1000$ equally spaced quadrature points, and different $m$ values.
			\label{tab:numerical-results-2}			
		}
		\begin{tabular}{|c|c|c|c|}
			\hline
			\multicolumn{1}{ |c| }{ $ m $ } &	
			\multicolumn{1}{ |c| }{ $\displaystyle  \frac{1}{10} \int_{0}^{10} \mathrm{e}_m(u(s))ds$ }  &
			\multicolumn{1}{ |c| }{ $\displaystyle  \frac{1}{10} \int_{0}^{10} \mathcal{E}_m(u(s)) \,ds$ }  \\		    
			\hline
			3 &   3.5296e-03 & 3.0515e-06 \\ 
			\hline
			6 &  1.1402e-03 & 8.8746e-06 \\ 
			\hline
			9 &  2.6359e-04 & 1.5090e-05  \\ 
			\hline
			15 &  1.1126e-05 & 1.1156e-05 \\ 
			\hline
			20 &   9.5913e-07 & 1.5183e-06 \\ 
			\hline
			25 &   1.8140e-07 & 1.8704e-07 \\ 
			\hline
			30 &   4.8765e-08 & 1.3491e-08 \\ 
			\hline
			35 &   1.0291e-09 & 9.7629e-10 \\ 
			\hline
			40 &   2.4680e-11 & 2.0416e-11 \\ 
			\hline			
		\end{tabular}
		\caption{
			Case 2: Time-averages of $\mathrm{e}_m(u)$ and $\mathcal{E}_{m}(u)$ for $d=41$, $\Delta t = 10^{-3}$, $10000$ equally spaced quadrature points, and different $m$ values.
			\label{tab:numerical-results-3}			
		}
\end{table}

\begin{table}[h]
		\begin{tabular}{|c|c|c|c|}
			\hline
			\multicolumn{1}{ |c| }{  $m $ } &		
			\multicolumn{1}{ |c| }{ $\displaystyle  \frac{1}{10} \int_{0}^{10} \mathrm{e}_m(u(s))ds$ }  &
			\multicolumn{1}{ |c| }{ $\displaystyle  \frac{1}{10} \int_{0}^{10} \mathcal{E}_m(u(s)) \,ds$ }  \\		    
			\hline
			3 &  3.5637e-03 & 2.7841e-06 \\ 
			\hline
			6 &   1.1664e-03 & 8.1109e-06 \\ 
			\hline
			9 &   2.7346e-04 & 1.3956e-05 \\ 
			\hline
			15 &   1.2084e-05 & 1.1937e-05 \\ 
			\hline
			20 &   1.1785e-06 & 1.6370e-06 \\ 
			\hline
			25 &  2.2727e-07 & 1.2913e-07 \\ 
			\hline
			30 &   6.2606e-08 &  9.5191e-09 \\ 
			\hline
			35 &   2.1106e-09 &  6.1538e-10 \\ 
			\hline
			40 &   1.0330e-10 & 2.0241e-11 \\ 
			\hline			
		\end{tabular}
		\caption{
			Case 2: Time-averages of $\mathrm{e}_m(u)$ and $\mathcal{E}_{m}(u)$ for $d=43$, $\Delta t = 10^{-4}$, $10000$ equally spaced quadrature points, and different $m$ values.
			\label{tab:numerical-results-4}			
		}
		\begin{tabular}{|c|c|c|c|}
			\hline
			\multicolumn{1}{ |c| }{  $m $ } &	
			\multicolumn{1}{ |c| }{ $\displaystyle  \frac{1}{10} \int_{0}^{10} \mathrm{e}_m(u(s))ds$ }  &
			\multicolumn{1}{ |c| }{ $\displaystyle  \frac{1}{10} \int_{0}^{10} \mathcal{E}_m(u(s)) \,ds$ }  \\		    
			\hline
			3 &  3.5668e-03 & 2.7594e-06 \\ 
			\hline
			6 & 1.1688e-03 & 8.0390e-06 \\ 
			\hline
			9 &  2.7436e-04 & 1.3843e-05 \\ 
			\hline
			15 &  1.2172e-05 & 1.2002e-05 \\ 
			\hline
			20 &   1.2030e-06 & 1.6405e-06 \\ 
			\hline
			25  & 2.3339e-07 & 1.2463e-07 \\ 
			\hline
			30  &  6.4120e-08 & 8.8971e-09 \\ 
			\hline
			35  &  2.2654e-09 & 5.6187e-10 \\ 
			\hline
			40  &  1.1211e-10 & 1.8216e-11 \\ 
			\hline			
		\end{tabular}
		\caption{
			Case 2: Time-averages of $\mathrm{e}_m(u)$ and $\mathcal{E}_{m}(u)$ for $d=43$, $\Delta t = 2*10^{-5}$, $10000$ equally spaced quadrature points, and different $m$ values.
			\label{tab:numerical-results-5}			
		}
\end{table}


\subsubsection{Case 3: }

In this case, we use an even longer time interval, i.e., $[0,T] = [0,100]$, and compare the time-averages for this time interval to those for the time intervals $[0,T] = [0,1]$ (Case 1) and $[0,T] = [0,10]$ (Case 2). 
For each time interval, we use the same time step values ($\Delta t=10^{-2}$) to generate the snapshots and all the subintervals in the composite trapezoidal rule utilized in the evaluation of the time average $\frac{1}{T}\int_0^{T} \mathrm{e}_m(u(s))$.
In Table~\ref{tab:numerical-results-6}, we list the time-averages of $\mathrm{e}_m(u)$ and $\mathcal{E}_{m}(u)$ for all three time intervals and different $m$ values.
We note that the time-averages of $\mathrm{e}_m(u)$ and $\mathcal{E}_{m}(u)$ are positive for all time intervals and $m$ values.
Furthermore, the magnitude of the time-average of $\mathcal{E}_m(u)$ is generally lower than the magnitude of the time-average of $\mathrm{e}_m(u)$.
Thus, we conclude that the time-average $\frac{1}{T}\int_0^{T} \mathrm{e}_m(u(s)) + \mathcal{E}_{m}(u(s)) \,ds$ in~\eqref{eqn:numerical-results-1} is positive for all time intervals and $m$ values.
Furthermore, we note that the time-averages of $\mathrm{e}_m(u)$ and $\mathcal{E}_{m}(u)$ for the time intervals $[0,T] = [0,100]$ and $[0,T] = [0,10]$ are close, whereas those for the time interval $[0,T] = [0,1]$ are slightly different.  
Thus, we conclude that the time interval $[0,T] = [0,10]$ is adequate for the approximation of the long time-average $\frac{1}{T}\int_0^{T} \mathrm{e}_m(u(s)) + \mathcal{E}_{m}(u(s)) \,ds$.   


\begin{table}[h]
	\begin{center}
		\begin{tabular}{|c|c|c|c|c|}
			\hline
			\multicolumn{1}{ |c| }{  $m$  } &
			\multicolumn{1}{ |c| }{ $\displaystyle  \frac{1}{100} \int_{0}^{100} \mathrm{e}_m(u(s)) \, ds$ }  &
			\multicolumn{1}{ |c|}{ $\displaystyle  \frac{1}{10} \int_{0}^{10} \mathrm{e}_m(u(s)) \, ds$ }  &					
			\multicolumn{1}{ |c| }{ $\displaystyle   \int_{0}^{1} \mathrm{e}_m(u(s)) \, ds$ }  \\		    
			\hline
			3 &   3.3687e-04 & 3.3683e-04 &  1.7634e-04  \\ 
			\hline
			6 &  1.0001e-04 & 1.0001e-04 & 7.3142e-05  \\ 
			\hline
			9 & 2.1146e-05 & 2.1147e-05 &  1.6701e-05   \\ 
			\hline
			15 &  5.5621e-07 & 5.5742e-07 &  4.7129e-07   \\ 
			\hline
			20 &  7.9605e-08 & 7.9605e-08 & 6.3799e-08   \\ 
			\hline
			25 &   8.3483e-09 & 8.3489e-09  &  5.7426e-09  \\ 
			\hline
			30 &   1.0839e-10 & 1.0840e-10  & 9.2430e-11   \\ 
			\hline
			35 &   2.7710e-12 & 2.7718e-12  &  2.9575e-12  \\ 
			\hline			
		\end{tabular}
		\caption{Case 3: Time-averages of $\mathrm{e}_m(u)$ for $d=36$, $\Delta t = 10^{-2}$, different $m$ values, and all subintervals used in the composite trapezoidal rule.
			\label{tab:numerical-results-6}			
		}
	\end{center}
\end{table}

\begin{table}[h]
	\begin{center}
		\begin{tabular}{|c|c|c|c|}
			\hline
			\multicolumn{1}{ |c| }{  $m $ } &
			\multicolumn{1}{ |c| }{ $\displaystyle  \frac{1}{100} \int_{0}^{100} \mathcal{E}_m(u(s)) \, ds$ }  &
			\multicolumn{1}{ |c|}{ $\displaystyle  \frac{1}{10} \int_{0}^{10} \mathcal{E}_m(u(s))   \,ds$ }  &					
			\multicolumn{1}{ |c| }{ $\displaystyle   \int_{0}^{1} \mathcal{E}_m(u(s)) \,ds$ }    \\		    
			\hline
			3  & 8.6270e-06  & 7.0320e-06 &     8.1057e-05  \\ 
			\hline
			6  & 2.1568e-05  & 2.1520e-05 &     3.2405e-05 \\ 
			\hline
			9  & 1.8614e-05  & 1.8599e-05 &  2.0255e-05 \\ 
			\hline
			15  & 5.5718e-07 & 5.4232e-07 &    5.8207e-07 \\ 
			\hline
			20  & 5.6044e-08  & 5.5874e-08 &  6.1423e-08 \\ 
			\hline
			25  & 1.0936e-09 & 9.9763e-10 &    5.9386e-10 \\ 
			\hline
			30  & 3.3046e-11 & 3.2478e-11 &    4.0949e-11  \\ 
			\hline
			35  & 6.4170e-13  & 6.3739e-13 &    1.2125e-12 \\ 
			\hline			
		\end{tabular}
		\caption{
			Case 3: Time-averages of $\mathcal{E}_{m}(u)$ for $d=36$, $\Delta t = 10^{-2}$, different $m$ values, and all subintervals used in the composite trapezoidal rule.
			\label{tab:numerical-results-6b}			
		}
	\end{center}
\end{table}

\subsubsection{Case 4:}

In this case, we use a much shorter time interval, i.e., $[0,T] = [0,0.1]$, and compare the time-averages for this time interval to those for the time intervals $[0,T] = [0,1], [0,T] = [0,10],$ and $ [0,T] = [0,100]$ (Case 3). 
We use two different time step values to generate the snapshots, but the same (i.e., $n=5000$) equally spaced subintervals in the composite trapezoidal rule utilized in the evaluation of the time average $\frac{1}{T}\int_0^{T} \mathrm{e}_m(u(s))$.
In Tables~\ref{tab:numerical-results-7}--\ref{tab:numerical-results-8}, we list the time-averages of $\mathrm{e}_m(u)$ and $\mathcal{E}_{m}(u)$ for two different time step values and different $m$ values.
We emphasize that, this time, the time-average of $\mathrm{e}_m(u)$ is negative for some $m$ values.
Furthermore, the magnitude of the time-average of $\mathcal{E}_m(u)$ is this time larger than the magnitude of the time-average of $\mathrm{e}_m(u)$.
This is in stark contrast with the previous cases.


\begin{table}[h]
		\begin{tabular}{|c|c|c|c|}
			\hline
			\multicolumn{1}{ |c| }{ $ m$  } &		
			\multicolumn{1}{ |c| }{ $\displaystyle  \frac{1}{0.1} \int_{0}^{0.1} \mathrm{e}_m(u(s))  ds$ }  &
			\multicolumn{1}{ |c| }{$\displaystyle \frac{1}{0.1} \int_{0}^{0.1} \mathcal{E}_m(u(s)) ds$ }   \\		    
			\hline
			3  & -6.8687e-04 & 2.7378e-05  \\ 
			\hline
			5  & -2.6333e-05 & 1.7795e-05 \\ 
			\hline
			7 & -9.1458e-07 & 4.2432e-06 \\ 
			\hline
			9  & -9.8188e-09  & 4.1220e-07 \\ 
			\hline
			13  & 2.3800e-10 &  1.5749e-09 \\ 
			\hline
			15  & 8.5423e-12 & 4.8043e-11 \\ 
			\hline			
		\end{tabular}
		\caption{
			Case 4: Time-averages of $\mathrm{e}_m(u)$ and $\mathcal{E}_{m}(u)$ for $\Delta t = 2*10^{-5}$, different $m$ values, $d=18$, and $5000$ equally spaced subintervals used in the composite trapezoidal rule.
			\label{tab:numerical-results-7}			
		}
		\begin{tabular}{|c|c|c|c|}
			\hline
			\multicolumn{1}{ |c| }{  $m$  } &	
			\multicolumn{1}{ |c| }{ $\displaystyle  \frac{1}{0.1} \int_{0}^{0.1} \mathrm{e}_m(u(s))  ds$ }  &
			\multicolumn{1}{ |c| }{$\displaystyle \frac{1}{0.1} \int_{0}^{0.1} \mathcal{E}_m(u(s))  ds$ }  \\		    
			\hline
			3 &  -6.8807e-04 & 2.7338e-05 \\ 
			\hline
			5 &  -2.6447e-05 & 1.7812e-05 \\ 
			\hline
			7 &  -9.1691e-07 & 4.2755e-06 \\ 
			\hline
			9 &  -9.5609e-09 & 4.1694e-07 \\ 
			\hline
			13 &  2.5287e-10  & 1.5941e-09  \\ 
			\hline
			15 &  1.0725e-11 & 4.7030e-11 \\ 
			\hline						
		\end{tabular}
		\caption{
			Case 4: Time-averages of $\mathrm{e}_m(u)$ and $\mathcal{E}_{m}(u)$ for $\Delta t = 10^{-5}$, different $m$ values, $d=18$, and $5000$ equally spaced subintervals used in the composite trapezoidal rule.
			\label{tab:numerical-results-8}			
		}
\end{table}



\section{Conclusions}
	\label{sec:conclusions}

In this preliminary study, we investigated theoretically and numerically the time-average of the exchange of energy among modes of reduced order models (ROMs) of fluid flows.
In particular, we were interested in the statistical equilibrium problem, and especially in the long-time averaging of the ROM solutions.
The main goal of the paper was to deduce the possible forward and backward average transfer of the energy among ROM basis functions (modes). 
We considered two types of ROM modes: eigenfunctions of the Stokes operator and proper orthogonal decomposition (POD) modes.
In Theorem~\ref{thm:thm-ROM} and Theorem~\ref{thm:newtheorem}, we proved analytical results for both types of ROM modes and we highlighted the differences between them, especially those stemming from the lack of orthogonality of the gradients of the POD basis functions.

In Section~\ref{sec:numerical-results}, we investigated numerically whether the time-average energy exchange between POD modes (i.e., $\frac{1}{T}\int_0^{T} \mathrm{e}_m(u(s)) + \mathcal{E}_{m}(u(s)) \,ds$) in Theorem~\ref{thm:thm-ROM} is positive.
To this end, we used the one-dimensional Burgers equation as a mathematical model.
We utilized a piecewise linear FE spatial discretization and a backward Euler temporal discretization.
To compute the time-averages, we used the composite trapezoidal rule. 
We tested different time steps, different number of subintervals in the composite trapezoidal rule, and, most importantly, different time intervals, to ensure that the computed quantities are indeed approximations of the time-averages and not numerical artifacts. 
The main conclusion of our numerical study is that, for long enough time intervals (i.e., time intervals longer than $[0,T] = [0,10]$), the time-average $\frac{1}{T}\int_0^{T} \mathrm{e}_m(u(s)) + \mathcal{E}_{m}(u(s)) \,ds$ in~\eqref{eqn:numerical-results-1} is positive.
Furthermore, the magnitude of the time-average of $\mathcal{E}_m(u)$ is much lower than the magnitude of the time-average of $\mathrm{e}_m(u)$.

There are several research directions that we plan to pursue.
Probably the most important one is the numerical investigation of the theoretical results in three-dimensional, high Reynolds number flows, which could shed new light on the energy transfer among ROM modes. 
A related, but different numerical investigation was performed in~\cite{couplet2003intermodal}.

\section*{Acknowledgments}
Luigi C. Berselli and Roger Lewandowski for the research that led to the present paper
were partially supported by a grant of the group GNAMPA of INdAM and by the project the
University of Pisa within the grant PRA$\_{}2018\_{}52$~UNIPI \textit{Energy and
  regularity: New techniques for classical PDE problems.}  Traian Iliescu and Birgul Koc
gratefully acknowledge the support provided by the NSF DMS-1821145 grant.

%

\def\ocirc#1{\ifmmode\setbox0=\hbox{$#1$}\dimen0=\ht0 \advance\dimen0
  by1pt\rlap{\hbox to\wd0{\hss\raise\dimen0
  \hbox{\hskip.2em$\scriptscriptstyle\circ$}\hss}}#1\else {\accent"17 #1}\fi}
  \def\cprime{$'$} \def\polhk#1{\setbox0=\hbox{#1}{\ooalign{\hidewidth
  \lower1.5ex\hbox{`}\hidewidth\crcr\unhbox0}}} \def\cprime{$'$}


\begin{thebibliography}{100}

\bibitem{Bat1953}
 Batchelor, G K (1953)
\newblock {\em The theory of homogeneous turbulence}.
\newblock Cambridge Monographs on Mechanics and Applied Mathematics. Cambridge
  University Press.

\bibitem{BIL2006}
Berselli L C, Iliescu, T, Layton W J, (2006)
\newblock {\em Mathematics of {L}arge {E}ddy {S}imulation of turbulent flows}.
\newblock Scientific Computation. Springer-Verlag, Berlin.

\bibitem{BL2018}
Berselli L C, and Lewandowski R, (2018)
\newblock On the {R}eynolds time-averaged equations and the long-time behavior
  of {L}eray-{H}opf weak solutions, with applications to ensemble averages. 
\newblock Technical Report \href{https://arxiv.org/abs/1801.08721}{ArXiv:1801.08721} ,

\bibitem{BFS2018}
Berselli L C, Fagioli S, Spirito S,  (2018)
\newblock Suitable weak solutions of the {N}avier-{S}tokes equations
  constructed by a space-time numerical discretization. 
\newblock {\em J. Math. Pures Appl. (9)}, online first
\href{https://doi.org/10.1016/j.matpur.2018.06.005}{doi:10.1016/j.matpur.2018.06.005} 

\bibitem{CL2014}
Chac{\'o}n~Rebollo T, Lewandowski R, (2014)
\newblock {\em Mathematical and numerical foundations of turbulence models and
  applications}.
\newblock Modeling and Simulation in Science, Engineering and Technology.
  Birkh\"auser/Springer, New York.

\bibitem{CF1988}
Constantin P, Foias C, (1988)
\newblock {\em Navier-{S}tokes equations}.
\newblock Chicago Lectures in Mathematics. University of Chicago Press,
  Chicago, IL.

\bibitem{couplet2003intermodal}
Couplet M, Sagaut P, Basdevant C, (2003)
\newblock {Intermodal energy transfers in a proper orthogonal
              decomposition--{G}alerkin representation of a turbulent
              separated flow}.
\newblock {\em J. Fluid Mech.}, 491:275--284.

\bibitem{DLM2017}
DeCaria V,  Layton W J, McLaughlin M, (2017)
\newblock A conservative, second order, unconditionally stable artificial
  compression method.
\newblock {\em Comput. Methods Appl. Mech. Engrg.}, 325:733--747.

\bibitem{GR86}
 Girault V, Raviart P A, (1986) 
  \newblock {\em Finite element methods for {N}avier-{S}tokes equations.}
  \newblock Springer Series in Computational Mathematics, Springer-Verlag, Berlin

\bibitem{Foi197273}
Foia\c{s} C,. (1972/73)
\newblock Statistical study of {N}avier-{S}tokes equations. {I}, {II}.
\newblock {\em Rend. Sem. Mat. Univ. Padova}, 48:219--348; ibid. 49:9--123.

\bibitem{FMRT2001a}
Foias C, Manley O, Rosa R, Temam R, (2001)
\newblock {\em Navier-{S}tokes equations and turbulence}, volume~83 of {\em
  Encyclopedia of Mathematics and its Applications}.
\newblock Cambridge University Press, Cambridge.


\bibitem{GMS2006}
Guermond J-L, Minev P, J.~Shen, (2006)
\newblock An overview of projection methods for incompressible flows.
\newblock {\em Comput. Methods Appl. Mech. Engrg.}, 195(44-47):6011--6045.

\bibitem{GOP2004}
Guermond J-L,  Oden J T, Prudhomme S, (2004)
\newblock Mathematical perspectives on large eddy simulation models for
  turbulent flows.
\newblock {\em J. Math. Fluid Mech.}, 6(2):194--248.

\bibitem{HRS2016}
Hesthaven J S, Rozza G, Stamm B, (2016)
\newblock {\em Certified reduced basis methods for parametrized partial
  differential equations}.
\newblock SpringerBriefs in Mathematics. Springer, Cham; BCAM Basque Center for
  Applied Mathematics, Bilbao, 
\newblock BCAM SpringerBriefs.


\bibitem{iliescu2014are}
Iliescu T, Wang Z, (2014)
\newblock Are the snapshot difference quotients needed in the proper orthogonal decomposition?
\newblock {\em SIAM J. Sci. Comput.}, 36(3):A1221--A1250.

\bibitem{JL2016} 
Jiang N, Layton W J, (2016) 
\newblock Algorithms and models for turbulence not at
  statistical equilibrium.  
\newblock {\em Comput. Math. Appl.}, 71(11):2352--2372.

\bibitem{Kol1941}
Kolmogorov A N, (1941)
\newblock The local structure of turbulence in incompressible viscous fluids
  for very large {R}eynolds number.
\newblock {\em Dokl. Akad. Nauk SSR}, 30:9--13.

\bibitem{LMQR2014}
Lassila T, Manzoni A, Quarteroni A, Rozza G, (2014)
\newblock Model order reduction in fluid dynamics: challenges and perspectives.
\newblock In {\em Reduced order methods for modeling and computational
  reduction}, volume~9 of {\em MS\&A. Model. Simul. Appl.}, pages 235--273.
  Springer, Cham.

\bibitem{Lay2014}
Layton W J, (2014)
\newblock The 1877 {B}oussinesq conjecture: {Tu}rbulent fluctuations are
  dissipative on the mean flow.
\newblock Technical Report \href{https://www.mathematics.pitt.edu/sites/default/files/ExpandedBoussinesqHypothesis.pdf}{TR-MATH 14-07}, Pittsburgh Univ.

\bibitem{LR2012}
Layton W J, Rebholz L, (2012)
\newblock {\em {A}pproximate {D}econvolution {M}odels of {T}urbulence
  {A}pproximate {D}econvolution {M}odels of {T}urbulence}, volume 2042 of {\em
  Lecture Notes in Mathematics}.
\newblock Springer, Heidelberg.

\bibitem{Lew2015}
Lewandowski R, (2015)
\newblock Long-time turbulence model deduced from the {N}avier-{S}tokes
  equations.
\newblock {\em Chin. Ann. Math. Ser. B}, 36(5):883--894.

\bibitem{JLL69} 
Lions, J-L, (1969)
  \newblock{ \em Quelques m\'ethodes de r\'esolution des probl\`emes aux limites non
    lin\'eaires.}
  \newblock   Dunod, Paris. 

\bibitem{MNRR1996}
M{\'a}lek J, Ne{\v{c}}as J, Rokyta M, R{\r u}{\v{z}}i{\v{c}}ka M, (1996)
\newblock {\em Weak and {M}easure-valued {S}olutions to {E}volutionary
  {P}{D}{E}s}, volume~13 of {\em Applied Mathematics and Mathematical
  Computations}.
\newblock Chapman \& Hall, London.

\bibitem{Pra1925}
{Prandtl} L, (1925)
\newblock {Bericht \"uber Untersuchungen zur ausgebildeten Turbulenz.}
\newblock {\em {Z. Angew. Math. Mech.}}, 5:136--139.

\bibitem{Pro1960}
{Prodi} G, (1960)
\newblock {Teoremi ergodici per le equazioni della idrodinamica.}
\newblock {C.I.M.E., Sistemi dinamici e Teoremi ergodici 16 p.}

\bibitem{Pro1961}
Prodi G, (1961)
\newblock On probability measures related to the {N}avier-{S}tokes equations in
  the 3-dimensional case.
\newblock Technical Report~2, Univ. Trieste.
\newblock Air Force Res. Div. Contract AF61(052)-414.

\bibitem{QMN2016}
Quarteroni A, Manzoni A, Negri F, (2016)
\newblock {\em Reduced basis methods for partial differential equations},
  volume~92 of {\em Unitext}.
\newblock Springer, Cham.
\newblock An introduction, La Matematica per il 3+2.

\bibitem{QRM2011}
Quarteroni A, Rozza G, Manzoni A, (2011)
\newblock Certified reduced basis approximation for parametrized partial
  differential equations and applications.
\newblock {\em J. Math. Ind.}, 1:Art. 3, 44.

\bibitem{Rey1895}
Reynolds O, (1895)
\newblock On the dynamic theory of the incompressible viscous fluids and the
  determination of the criterion.
\newblock {\em Philos. Trans. Roy. Soc. London Ser. A}, 186:123--164.

\bibitem{Roz2014}
Rozza G, (2014)
\newblock Fundamentals of reduced basis method for problems governed by
  parametrized {PDE}s and applications.
\newblock In {\em Separated representations and {PGD}-based model reduction},
  volume 554 of {\em CISM Courses and Lect.}, pages 153--227. Springer, Vienna.

\bibitem{Sag2001}
Sagaut P, (2001)
\newblock {\em Large eddy simulation for incompressible flows}.
\newblock Scientific Computation. Springer-Verlag, Berlin.
\newblock An introduction, With an introduction by Marcel Lesieur, Translated
  from the 1998 French original by the author.

\bibitem{WWXI2017a}
Wells D, Wang Z, Xie X, Iliescu T, (2017)
\newblock An evolve-then-filter regularized reduced order model for
  convection-dominated flows.
\newblock {\em Internat. J. Numer. Methods Fluids}, 84(10):598--615.

\bibitem{XWWI2017b}
Xie X, Wells D, Wang Z, Iliescu T, (2017)
\newblock Approximate deconvolution reduced order modeling.
\newblock {\em Comput. Methods Appl. Mech. Engrg.}, 313:512--534.

\bibitem{xie2018data}
Xie X, Mohebujjaman M, Rebholz L G,  Iliescu T, (2018)
\newblock Data-driven filtered reduced order modeling of fluid flows.
\newblock {\em SIAM J. Sci. Comput.}, 40(3):B834--B857.

\bibitem{xie2018lagrangian}
Xie X, Mohebujjaman M, Rebholz L G,  Iliescu T, (2018)
\newblock Lagrangian data-driven reduced order modeling of finite time {L}yapunov exponents.
\newblock arXiv preprint, \url{http://arxiv.org/abs/1808.05635}.

\end{thebibliography}

%
%
%
%
%
%
%
%
%
%
%
%
%
%
\end{document}